\documentclass[a4paper,11pt]{article}
\usepackage{amsmath,amsthm}
\usepackage{aligned-overset}
\usepackage{authblk}
\usepackage[style=numeric-comp,maxbibnames=99,bibencoding=utf8,giveninits=true]{biblatex}
\usepackage{booktabs}
\usepackage{cancel}
\usepackage{cases}
\usepackage{comment}
\usepackage[hmargin={26mm,26mm},vmargin={30mm,35mm}]{geometry}
\usepackage{hyperref}
\usepackage[compact]{titlesec}
\usepackage{tikz-cd}

\usepackage{xcolor}

\usepackage{newtxtext}
\usepackage{newtxmath}

\bibliography{ddr-poincare}

\newtheorem{theorem}{Theorem}

\newtheorem{lemma}[theorem]{Lemma}

\theoremstyle{remark}

\theoremstyle{definition}

\theoremstyle{example}

\newcommand{\email}[1]{\href{mailto:#1}{#1}}

\DeclareMathOperator{\card}{card}
\DeclareMathOperator{\diam}{diam}
\DeclareMathOperator{\sign}{sgn}

\newcommand{\Real}{\mathbb{R}}

\newcommand{\st}{\,:\,}

\DeclareMathOperator{\DIV}{div}
\DeclareMathOperator{\VDIV}{\bf div}
\DeclareMathOperator{\CURL}{\bf curl}
\DeclareMathOperator{\GRAD}{\bf grad}
\DeclareMathOperator{\ROT}{rot}
\DeclareMathOperator{\VROT}{\bf rot}

\newcommand{\Hdiv}[1]{\bvec{H}(\DIV;#1)}
\newcommand{\Hcurl}[1]{\bvec{H}(\CURL;#1)}

\DeclareMathOperator{\Ker}{Ker}

\newcommand{\sphere}{\mathcal{S}}

\newcommand{\compl}{\mathrm{c}}

\newcommand{\Poly}[2][]{\mathcal{P}_{#1}^{#2}}
\newcommand{\vPoly}[2][]{\cvec{P}_{#1}^{#2}}
\newcommand{\Roly}[1]{\cvec{R}^{#1}}
\newcommand{\cRoly}[1]{\cvec{R}^{\compl,#1}}
\newcommand{\Goly}[1]{\cvec{G}^{#1}}
\newcommand{\cGoly}[1]{\cvec{G}^{\compl, #1}}

\newcommand{\lproj}[2]{\pi_{\mathcal{P},#2}^{#1}}

\newcommand{\Rproj}[2]{\bvec{\pi}_{\cvec{R},#2}^{#1}}
\newcommand{\cRproj}[2]{\bvec{\pi}_{\cvec{R},#2}^{\compl,#1}}
\newcommand{\cXproj}[2]{\bvec{\pi}_{\cvec{X},#2}^{\compl,#1}}

\newcommand{\Gproj}[2]{\bvec{\pi}_{\cvec{G},#2}^{#1}}
\newcommand{\cGproj}[2]{\bvec{\pi}_{\cvec{G},#2}^{\compl,#1}}

\newcommand{\RT}[1]{\mathcal{RT}^{#1}}
\newcommand{\NE}[1]{\mathcal{N}^{#1}}

\newcommand{\RTproj}[2]{\bvec{\pi}_{\mathcal{RT},#2}^{#1}}

\DeclareRobustCommand{\bvec}[1]{\boldsymbol{#1}}
\newcommand{\uvec}[1]{\underline{\bvec{#1}}}

\newcommand{\cvec}[1]{\bvec{\mathcal{#1}}}

\newcommand{\jump}[2]{\llbracket#2\rrbracket_{#1}}

\newcommand{\normal}{\bvec{n}}
\newcommand{\tangent}{\bvec{t}}

\newcommand{\elements}[1]{\mathcal{T}_{#1}}
\newcommand{\faces}[1]{\mathcal{F}_{#1}}
\newcommand{\edges}[1]{\mathcal{E}_{#1}}
\newcommand{\vertices}[1]{\mathcal{V}_{#1}}

\newcommand{\Mh}{\mathcal{M}_h}
\newcommand{\Th}{\elements{h}}
\newcommand{\Fh}{\faces{h}}
\newcommand{\Eh}{\edges{h}}
\newcommand{\Vh}{\vertices{h}}

\newcommand{\FT}{\faces{T}}
\newcommand{\ET}{\edges{T}}
\newcommand{\VT}{\vertices{T}}

\newcommand{\EF}{\edges{F}}

\newcommand{\VE}{\vertices{E}}

\newcommand{\subelements}[1]{\mathfrak{T}_{#1}}
\newcommand{\subfaces}[1]{\mathfrak{F}_{#1}}
\newcommand{\subedges}[1]{\mathfrak{E}_{#1}}
\newcommand{\subvertices}[1]{\mathfrak{V}_{#1}}

\newcommand{\fMh}{\mathfrak{M}_h}
\newcommand{\fTh}{\subelements{h}}
\newcommand{\fFh}{\subfaces{h}}
\newcommand{\fEh}{\subedges{h}}
\newcommand{\fVh}{\subvertices{h}}

\newcommand{\FS}{\subfaces{S}}
\newcommand{\VS}{\subvertices{S}}
\newcommand{\ST}{\subelements{T}}
\newcommand{\ES}{\subedges{S}}

\newcommand{\nF}{\normal_F}

\newcommand{\wFE}{\omega_{FE}}
\newcommand{\wTF}{\omega_{TF}}
\newcommand{\wSF}{\omega_{SF}}
\newcommand{\wOF}{\omega_{\Omega F}}

\newcommand{\norm}[2]{\|#2\|_{#1}}
\newcommand{\seminorm}[2]{|#2|_{#1}}
\newcommand{\vvvert}{\vert\kern-0.25ex\vert\kern-0.25ex\vert}
\newcommand{\tnorm}[2]{\vvvert #2\vvvert_{#1}}

\newcommand{\trimmed}{{-}}

\newcommand{\Xgrad}[1]{\underline{X}_{\GRAD,#1}^k}
\newcommand{\Xcurl}[1]{\uvec{X}_{\CURL,#1}^k}
\newcommand{\Xdiv}[1]{\uvec{X}_{\DIV,#1}^k}
\newcommand{\HXdiv}{\uvec{\mathfrak{H}}_{\DIV,h}^k}

\newcommand{\uGh}{\uvec{G}_h^k}
\newcommand{\GT}{\bvec{G}_T^k}
\newcommand{\GF}{\bvec{G}_F^k}
\newcommand{\GE}{G_E^k}

\newcommand{\uCh}{\uvec{C}_h^k}
\newcommand{\uCT}{\uvec{C}_T^k}
\newcommand{\CF}{C_F^k}
\newcommand{\CT}{\boldsymbol{C}_T^k}

\newcommand{\Dh}{D_h^k}
\newcommand{\DT}{D_T^k}

\newcommand{\trF}{\gamma_F^{k+1}}
\newcommand{\trFt}{\bvec{\gamma}_{{\rm t},F}^k}

\begin{document}

\title{Uniform Poincar\'e inequalities for the Discrete de Rham complex on general domains}
\author{Daniele A. Di Pietro}
\author{Marien-Lorenzo Hanot}
\affil{IMAG, Univ Montpellier, CNRS, Montpellier, France, \email{daniele.di-pietro@umontpellier.fr}, \email{marien-lorenzo.hanot@umontpellier.fr}}

\maketitle


\begin{abstract}
  In this paper we prove Poincar\'e inequalities for the Discrete de Rham (DDR) sequence on a general connected polyhedral domain $\Omega$ of $\Real^3$.
  We unify the ideas behind the inequalities for all three operators in the sequence, deriving new proofs for the Poincar\'e inequalities for the gradient and the divergence, and extending the available Poincar\'e inequality for the curl to domains with arbitrary second Betti numbers.
  A key preliminary step consists in deriving ``mimetic'' Poincaré inequalities giving the existence and stability of the solutions to topological balance problems useful in general discrete geometric settings.
  As an example of application, we study the stability of a novel DDR scheme for the magnetostatics problem on domains with general topology.
\end{abstract}


\section{Introduction}

Poincaré inequalities are a key tool to prove the well-posedness of many common partial differential equation problems.
Mimicking them at the discrete level is typically required for the stability of numerical approximations.
Poincar\'e inequalities for conforming Finite Element de Rham complexes can be derived through bounded cochain projections 
as described, e.g., in \cite[Chapter~5]{Arnold:18}; see also \cite{Christiansen.Licht:20} for a recent generalisation.
In the context of Virtual Element de Rham complexes \cite{Beirao-da-Veiga.Brezzi.ea:18}, 
similar results typically hinge on non-trivial norm comparison results, 
examples of which can be found in \cite{Beirao-da-Veiga.Dassi.ea:22*1}.
Discrete Poincar\'e-type inequalities in the context of the (non-compatible) Hybrid High-Order methods have been derived, 
e.g., in \cite{Di-Pietro.Droniou:20} (gradient), \cite{Botti.Di-Pietro.ea:17} (symmetric gradient) 
and \cite{Chave.Di-Pietro.ea:22,Lemaire.Pitassi:23} (curl).

The focus of the present work is on the derivation of Poincar\'e inequalities for the Discrete de Rham (DDR) sequence of \cite{Di-Pietro.Droniou:23*1} on domains with general topology.
Unlike Finite and Virtual Elements, DDR formulations are fully discrete, 
with spaces spanned by vectors of polynomials and continuous vector calculus operators replaced by discrete counterparts.
Discrete Poincar\'e inequalities thus require to bound $L^2$-like norms of vectors of polynomials with $L^2$-like norms of suitable discrete operators applied to them.
To establish such bounds, we take inspiration from \cite{Di-Pietro.Droniou.ea:23}, 
where it was noticed that the topological information is fully contained in the lowest-order DDR subsequence, 
and \cite{Di-Pietro.Droniou:21}, where a Poincar\'e inequality for the curl on topologically trivial domains of $\Real^3$ was derived.
The lowest-order DDR sequence is strongly linked to Mimetic Finite Differences and related methods \cite{Brezzi.Lipnikov.ea:05,Brezzi.Buffa.ea:09,Beirao-da-Veiga.Lipnikov.ea:14,Bonelle.Ern:14,Bonelle.Di-Pietro.ea:15,Codecasa.Specogna.ea:10}.
The first step to prove discrete Poincar\'e inequalities in DDR spaces is thus precisely to establish the mimetic counterparts stated in Theorems \ref{thm:Whitney.V}, \ref{thm:Whitney.E}, and \ref{thm:Whitney.F} below.
Their proofs require to work at the global level, with conditions accounting for the topology of the domain appearing for the curl.
The discrete Poincar\'e inequalities for arbitrary-order DDR spaces collected in Section \ref{sec:main.results} below 
are then obtained combining the mimetic Poincar\'e inequalities with local estimates of the higher-order components.

We next briefly discuss the links between the present work and previous results for DDR methods.
Fully general Poincaré inequalities for the gradient and the divergence had already been obtained, 
respectively, in \cite[Theorem 3]{Di-Pietro.Droniou:23*1} and \cite{Di-Pietro.Droniou:21} using different techniques.
The main novelty of the proofs provided here is that they are better suited to generalisations in the framework of the Polytopal Exterior Calculus recently introduced in \cite{Bonaldi.Di-Pietro.ea:23}.
A Poincar\'e inequality for the curl on topologically trivial domains had been obtained in \cite[Theorem 20]{Di-Pietro.Droniou:21}.
The main novelty with respect to this result consists in the extension to domains encapsulating voids.
The interest of the material in this paper is additionally that it contains preliminary 
results to establish discrete Poincar\'e inequalities for advanced complexes, 
such as the three-dimensional discrete div-div complex recently introduced in \cite{Di-Pietro.Hanot:23}.

The rest of the paper is organized as follows.
The definitions of the relevant DDR spaces and operators are briefly recalled in Section \ref{sec:DDRconstruction}.
Mimetic Poincarés inequalities are derived in Section \ref{sec:proof.poincaré}, and then used to prove discrete Poincarés inequalities for the DDR complex in Section \ref{sec:proof.DDR.poincare}.
The latter are used in Section \ref{sec:vectorLaplace} to carry out the stability analysis of a DDR scheme for the magnetostatics problem on domains with general topology.
Some arguments in the proofs of mimetic Poincar\'e inequalities rely on specific shape functions for Finite Element spaces on a submesh, whose definitions and properties are summarised in Appendix \ref{sec:simplicial.de-rham}.


\section{Discrete de Rham construction} \label{sec:DDRconstruction}

\subsection{Domain and mesh}\label{sec:setting:domain.mesh}

Let $\Omega\subset\Real^3$ denote a connected polyhedral domain.
We consider a polyhedral mesh $\Mh\coloneq\Th\cup\Fh\cup\Eh\cup\Vh$, where $\Th$ gathers the elements, $\Fh$ the faces, $\Eh$ the edges, and $\Vh$ the vertices.
For all $Y\in\Mh$, we denote by $h_Y$ its diameter and set $h\coloneq\max_{T\in\Th}h_T$.
For each face $F\in\Fh$, we fix a unit normal $\normal_F$ to $F$ and, for each edge $E\in\Eh$, a unit tangent $\tangent_E$. For $T\in\Th$, $\FT$ gathers the faces on the boundary $\partial T$ of $T$ and $\ET$ the edges in $\partial T$; if $F\in\Fh$, $\EF$ is the set of edges contained in the boundary $\partial F$ of $F$.
For $F\in\FT$, $\omega_{TF}\in\{-1,+1\}$ is such that $\omega_{TF}\normal_F$ is the outer normal on $F$ to $T$.

Each face $F\in\Fh$ is oriented counter-clockwise with respect to $\normal_F$ and, for $E\in\EF$, we let $\omega_{FE}\in\{-1,+1\}$ be such that $\omega_{FE}=+1$ if $\tangent_E$ points along the boundary $\partial F$ of $F$ in the clockwise sense, and $\omega_{FE}=-1$ otherwise; we also denote by $\normal_{FE}$ the unit normal vector to $E$, in the plane spanned by $F$, such that  $\omega_{FE}\normal_{FE}$ points outside $F$.
We denote by $\GRAD_F$ and $\DIV_F$ the tangent gradient and divergence operators acting on smooth enough functions.
Moreover, for any $r:F\to\Real$ and $\bvec{z}:F\to\Real^2$ smooth enough, we let
$\VROT_F r\coloneq (\GRAD_F r)^\perp$ and
$\ROT_F\bvec{z}=\DIV_F(\bvec{z}^\perp)$,
with $\perp$ denoting the rotation of angle $-\frac\pi2$ in the oriented tangent space to $F$.

We further assume that $(\Th,\Fh)$ belongs to a regular mesh sequence in the sense of \cite[Definition 1.9]{Di-Pietro.Droniou:20}, with mesh regularity parameter $\varrho>0$.
This implies that, for each $Y\in\Th\cup\Fh\cup\Eh$, there exists a point $\bvec{x}_{Y}\in Y$ such that the ball centered at $\bvec{x}_Y$ and of radius $\varrho h_Y$ is contained in $Y$.

Throughout the paper, $a\lesssim b$ (resp., $a\gtrsim b$) stands for $a\le Cb$ (resp., $a\ge Cb$) with $C$ depending only on $\Omega$, the mesh regularity parameter and, when polynomial functions are involved, the corresponding polynomial degree.
We also write $a \simeq b$ when both $a \lesssim b$ and $b \lesssim a$ hold.

\subsection{Polynomial spaces and $L^2$-orthogonal projectors}

For any $Y\in\Mh$ and an integer $\ell\ge 0$, we denote by $\Poly{\ell}(Y)$ the space spanned by the restriction to $Y$ of polynomial functions of the space variables.
Let, for $Y \in \Th \cup \Fh$, $\vPoly{\ell}(Y)\coloneq\Poly{\ell}(Y)^n$ with $n$ denoting the dimension of $Y$.
We have the following direct decompositions:
For all $F\in\Fh$, 
\[
\text{%
  $\vPoly{\ell}(F) = \Roly{\ell}(F) \oplus \cRoly{\ell}(F)$
  with $\Roly{\ell}(F)\coloneq\VROT_F\Poly{\ell+1}(F)$
  and $\cRoly{\ell}(F)\coloneq(\bvec{x}-\bvec{x}_F)\Poly{\ell-1}(F)$
}
\]
and, for all $T\in\Th$,
\[
\begin{alignedat}{4}
  \vPoly{\ell}(T)
  &= \Goly{\ell}(T) \oplus \cGoly{\ell}(T)
  &\enspace&
  \text{%
    with $\Goly{\ell}(T)\coloneq\GRAD\Poly{\ell+1}(T)$
    and $\cGoly{\ell}(T)\coloneq(\bvec{x}-\bvec{x}_T)\times \vPoly{\ell-1}(T)$%
  } 
  \\ \label{eq:vPoly=Roly+cRoly}
  &= \Roly{\ell}(T) \oplus \cRoly{\ell}(T)
  &\enspace&
  \text{%
    with $\Roly{\ell}(T)\coloneq\CURL\Poly{\ell+1}(T)$
    and $\cRoly{\ell}(T)\coloneq(\bvec{x}-\bvec{x}_T)\Poly{\ell-1}(T)$.%
  }
\end{alignedat}
\]
We extend the above notations to negative exponents $\ell$ by setting all the spaces appearing in the decompositions equal to the trivial vector space $\{\bvec{0}\}$.
Given a polynomial (sub)space $\mathcal{X}^\ell(Y)$ on $Y\in\Mh$, the corresponding $L^2$-orthogonal projector is denoted by $\pi_{\mathcal{X},Y}^\ell$.
Boldface font will be used when the elements of $\mathcal{X}^\ell(Y)$ are vector-valued, and $\cXproj{\ell}{Y}$ will denote the $L^2$-orthogonal projector on $\cvec{X}^{{\rm c},\ell}(Y)$.

\subsection{DDR spaces}\label{sec:DDRconstruction:spaces}

The discrete counterparts of the spaces appearing in the continuous de Rham complex are defined as follows:

\[
\Xgrad{h}\coloneq\Big\{
\begin{aligned}[t]
  \underline{q}_h
  &=\big((q_T)_{T\in\Th},(q_F)_{F\in\Fh}, (q_E)_{E\in\Eh}, (q_V)_{V\in\Vh}\big)\st
  \\
  &\qquad
  \text{$q_T\in \Poly{k-1}(T)$ for all $T\in\Th$,
    $q_F\in\Poly{k-1}(F)$ for all $F\in\Fh$,}
  \\
  &\qquad
  \text{$q_E\in\Poly{k-1}(E)$ for all $E\in\Eh$,      
    and $q_V \in \Real$ for all $V\in\Vh$}
  \Big\},
\end{aligned}
\]
\[
  \Xcurl{h}\coloneq\Big\{
  \begin{aligned}[t]
    \uvec{v}_h
    &=\big(
    (\bvec{v}_{\cvec{R},T},\bvec{v}_{\cvec{R},T}^\compl)_{T\in\Th},(\bvec{v}_{\cvec{R},F},\bvec{v}_{\cvec{R},F}^\compl)_{F\in\Fh}, (v_E)_{E\in\Eh}
    \big)\st
    \\
    &\qquad\text{$\bvec{v}_{\cvec{R},T}\in\Roly{k-1}(T)$ and $\bvec{v}_{\cvec{R},T}^\compl\in\cRoly{k}(T)$ for all $T\in\Th$,}
    \\
    &\qquad\text{$\bvec{v}_{\cvec{R},F}\in\Roly{k-1}(F)$ and $\bvec{v}_{\cvec{R},F}^\compl\in\cRoly{k}(F)$ for all $F\in\Fh$,}
    \\
    &\qquad\text{and $v_E\in\Poly{k}(E)$ for all $E\in\Eh$}\Big\},
  \end{aligned}
\]
\[
  \Xdiv{h}\coloneq\Big\{
  \begin{aligned}[t]
    \uvec{w}_h
    &=\big((\bvec{w}_{\cvec{G},T},\bvec{w}_{\cvec{G},T}^\compl)_{T\in\Th}, (w_F)_{F\in\Fh}\big)\st
    \\
    &\qquad\text{$\bvec{w}_{\cvec{G},T}\in\Goly{k-1}(T)$ and $\bvec{w}_{\cvec{G},T}^\compl\in\cGoly{k}(T)$ for all $T\in\Th$,}
    \\
    &\qquad\text{and $w_F\in\Poly{k}(F)$ for all $F\in\Fh$}
    \Big\},
  \end{aligned}
\]
and
\[
\Poly{k}(\Th)\coloneq\left\{
q_h\in L^2(\Omega)\st\text{$(q_h)_{|T}\in\Poly{k}(T)$ for all $T\in\Th$}
\right\}.
\]

\subsection{Local vector calculus operators and potentials}

\subsubsection{Gradient}

For any $E\in\Eh$, the edge gradient $\GE:\Xgrad{E}\to\Poly{k}(E)$ is such that, for all $\underline{q}_E\in\Xgrad{E}$,
\begin{equation}\label{eq:GE}
  \int_E \GE \underline{q}_E\, r
  = -\int_E q_E\, r' + \jump{E}{q_V\, r},
\end{equation}
with derivative taken in the direction of $\tangent_E$ and with
$\jump{E}{\cdot}$ denoting the difference between vertex values on an edge such that, for any function $\phi\in C^0(\overline{E})$ and any family $\{w_{V_1},w_{V_2}\}$ of vertex values such that $\tangent_E$ points from $V_1$ to $V_2$,
\[
\jump{E}{w_V\phi}\coloneq w_{V_2} \phi(\bvec{x}_{V_2}) - w_{V_1} \phi(\bvec{x}_{V_1}).
\]

For any $F\in\Fh$, the face gradient $\GF:\Xgrad{F}\to\vPoly{k}(F)$ and the scalar trace $\trF:\Xgrad{F}\to\Poly{k+1}(F)$ are such that, for all $\underline{q}_F\in\Xgrad{F}$,
\begin{alignat}{4}\label{eq:GF}
  \int_F\GF\underline{q}_F\cdot\bvec{v}
  &= -\int_F q_F\DIV_F\bvec{v}
  + \sum_{E\in\EF}\omega_{FE}\int_E q_E~(\bvec{v}\cdot\normal_{FE})
  &\quad&\forall\bvec{v}\in\vPoly{k}(F),
  \\ \nonumber
  \int_F\trF\underline{q}_F\DIV_F\bvec{v}
  &= -\int_F\GF\underline{q}_F\cdot\bvec{v}
  + \sum_{E\in\EF}\omega_{FE}\int_E q_E~(\bvec{v}\cdot\normal_{FE})
  &\quad&\forall\bvec{v}\in\cRoly{k+2}(F).
\end{alignat}

Similarly, for all $T\in\Th$, the element gradient $\GT:\Xgrad{T}\to\vPoly{k}(T)$ is defined such that, for all $\underline{q}_T\in\Xgrad{T}$,
\begin{equation}\label{eq:GT}
  \int_T\GT\underline{q}_T\cdot\bvec{v}
  = -\int_T q_T\DIV\bvec{v}
  + \sum_{F\in\FT}\omega_{TF}\int_F\trF\underline{q}_F~(\bvec{v}\cdot\normal_F)
  \qquad\forall\bvec{v}\in\vPoly{k}(T),
\end{equation}

\subsubsection{Curl}

For all $F\in\Fh$, the face curl $\CF:\Xcurl{F}\to\Poly{k}(F)$ and tangential trace $\trFt:\Xcurl{F}\to\vPoly{k}(F)$ are such that, for all $\uvec{v}_F\in\Xcurl{F}$,
\begin{equation}\label{eq:CF}
  \int_F\CF\uvec{v}_F\,r
  = \int_F\bvec{v}_{\cvec{R},F}\cdot\VROT_F r
  - \sum_{E\in\EF}\omega_{FE}\int_E v_E\,r
  \qquad\forall r\in\Poly{k}(F)
\end{equation}
and, for all $(r,\bvec{w})\in\Poly[0]{k+1}(F) \times\cRoly{k}(F)$,
\begin{equation*}
  \int_F\trFt\uvec{v}_F\cdot(\VROT_F r + \bvec{w})
  = \int_F\CF\uvec{v}_F\,r
  + \sum_{E\in\EF}\omega_{FE}\int_E v_E\,r
  + \int_F\bvec{v}_{\cvec{R},F}^\compl\cdot\bvec{w}.
\end{equation*}

For all $T\in\Th$, the element curl $\CT:\Xcurl{T}\to\vPoly{k}(T)$ is defined such that, for all $\uvec{v}_T\in\Xcurl{T}$,
\begin{equation}\label{eq:CT}
  \int_T\CT\uvec{v}_T\cdot\bvec{w}
  = \int_T\bvec{v}_{\cvec{R},T}\cdot\CURL\bvec{w}
  + \sum_{F\in\FT}\omega_{TF}\int_F\trFt\uvec{v}_F\cdot(\bvec{w}\times\normal_F)
  \qquad\forall\bvec{w}\in\vPoly{k}(T).
\end{equation}

\subsubsection{Divergence}

For all $T\in\Th$, the element divergence $\DT:\Xdiv{T}\to\Poly{k}(T)$ is defined by:
For all $\uvec{w}_T\in\Xdiv{T}$,
\begin{equation}\label{eq:DT}
  \int_T\DT\uvec{w}_T\,q
  = -\int_T\bvec{w}_{\cvec{G},T}\cdot\GRAD q
  + \sum_{F\in\FT}\omega_{TF}\int_F w_F\,q
  \qquad\forall q\in\Poly{k}(T).
\end{equation}

\subsection{DDR complex}\label{sec:ddr.complex}

The DDR complex reads:
\[
\begin{tikzcd}
  0\arrow{r}
  & \Xgrad{h}\arrow{r}{\uGh}
  & \Xcurl{h}\arrow{r}{\uCh}
  & \Xdiv{h}\arrow{r}{\Dh}
  & \Poly{k}(\Th)\arrow{r}{0}
  & \{0\},
\end{tikzcd}
\]
where, for all $(\underline{q}_h,\uvec{v}_h,\uvec{w}_h)\in\Xgrad{h}\times\Xcurl{h}\times\Xdiv{h}$,
\begin{align} \label{eq:uGh}
  \uGh\underline{q}_h
  &\coloneq
  \big(
  ( \Rproj{k-1}{T}\GT\underline{q}_T,\cRproj{k}{T}\GT\underline{q}_T )_{T\in\Th},
  ( \Rproj{k-1}{F}\GF\underline{q}_F,\cRproj{k}{F}\GF\underline{q}_F )_{F\in\Fh},
  ( \GE q_E )_{E\in\Eh}
  \big),
  \\ \label{eq:uCh}
  \uCh\uvec{v}_h
  &\coloneq\big(
  ( \Gproj{k-1}{T}\CT\uvec{v}_T,\cGproj{k}{T}\CT\uvec{v}_T )_{T\in\Th},
  ( \CF\uvec{v}_F )_{F\in\Fh}
  \big),
  \\ \nonumber
  (\Dh\uvec{w}_h)_{|T}
  &\coloneq\DT\uvec{w}_T\qquad\forall T\in\Th.
\end{align}

\subsection{Component norms}

We endow the discrete spaces defined in Section \ref{sec:DDRconstruction:spaces} with the $L^2$-like norms defined as follows:
For all $(\underline{q}_h,\uvec{v}_h,\uvec{w}_h)\in\Xgrad{h}\times\Xcurl{h}\times\Xdiv{h}$,
\[
\begin{alignedat}{4}
  \tnorm{\GRAD,h}{\underline{q}_h}^2
  &\coloneq \sum_{T\in\Th} \tnorm{\GRAD,T}{\underline{q}_T}^2\text{ with}
  \\ 
  \tnorm{\GRAD,T}{\underline{q}_T}^2
  &\coloneq \norm{L^2(T)}{q_T}^2
  + h_T \sum_{F\in\Fh} \tnorm{F}{\underline{q}_F}^2
  &\qquad& \forall T \in \Th,
  \\ 
  \tnorm{\GRAD,F}{\underline{q}_F}^2
  &\coloneq \norm{L^2(F)}{q_F}^2  
  + h_F \sum_{E\in\Eh} \tnorm{\GRAD,E}{\underline{q}_E}^2
  &\qquad& \forall F \in \Fh,
  \\ 
  \tnorm{\GRAD,E}{q_E}^2
  &\coloneq \norm{L^2(E)}{q_E}^2 + h_E \sum_{V\in\VE} |q_V|^2
  &\qquad& \forall E \in \Eh,
\end{alignedat}
\]
\begin{equation} \label{eq:tnorm.curl.h}
  \begin{alignedat}{4} 
    \tnorm{\CURL,h}{\uvec{v}_h}^2
    &\coloneq \sum_{T\in\Th} \tnorm{\CURL,T}{\uvec{v}_T}^2\text{ with}
    \\
    \tnorm{\CURL,T}{\uvec{v}_T}^2
    &\coloneq
    \norm{\bvec{L}^2(T;\Real^3)}{\bvec{v}_{\cvec{R},T}}^2
    + \norm{\bvec{L}^2(T;\Real^3)}{\bvec{v}_{\cvec{R},T}^\compl}^2
    + h_T \sum_{F\in\FT} \tnorm{\CURL,F}{\uvec{v}_F}^2
    &\qquad& \forall T \in \Th,
    \\
    \tnorm{\CURL,F}{\uvec{v}_F}^2
    &\coloneq
    \norm{\bvec{L}^2(F;\Real^2)}{\bvec{v}_{\cvec{R},F}}^2
    + \norm{\bvec{L}^2(F;\Real^2)}{\bvec{v}_{\cvec{R},F}^\compl}^2
    + h_F \sum_{E\in\EF} \norm{L^2(E)}{v_E}^2
    &\qquad& \forall F \in \Fh,
  \end{alignedat}
\end{equation}
and
\begin{equation} \label{eq:tnorm.div.h}
  \begin{aligned} 
    \tnorm{\DIV,h}{\uvec{w}_h}^2
    &\coloneq \sum_{T\in\Th} \tnorm{\DIV,T}{\uvec{w}_T}^2\text{ with}
    \\
    \tnorm{\DIV,T}{\uvec{w}_T}^2
    &\coloneq
    \norm{\bvec{L}^2(T;\Real^3)}{\bvec{w}_{\cvec{G},T}}^2
    + \norm{\bvec{L}^2(T;\Real^3)}{\bvec{w}_{\cvec{G},T}^\compl}^2
    + h_T \sum_{F\in\FT} \norm{L^2(F)}{v_F}^2
    \qquad \forall T \in \Th.
  \end{aligned}
\end{equation}

\subsection{Main results}\label{sec:main.results}

\begin{theorem}[Poincar\'e inequality for the gradient]\label{thm:uGh.poincare}
  For all $\underline{p}_h \in \Xgrad{h}$, it holds
  \[ 
    \inf_{\underline{r}_h \in \Ker \uGh} \tnorm{\GRAD,h}{\underline{p}_h - \underline{r}_h}
    \lesssim \tnorm{\CURL,h}{\uGh \underline{p}_h},
    \] 
  with hidden constant only depending on $\Omega$, the mesh regularity parameter, and $k$.
\end{theorem}

\begin{proof}
  See Section \ref{sec:uGh.poincare}.
\end{proof}

\begin{theorem}[Poincar\'e inequality for the curl]\label{thm:uCh.poincare}
  For all $\uvec{v}_h \in \Xcurl{h}$, it holds
  \[ 
    \inf_{\uvec{z}_h \in \Ker \uCh} \tnorm{\CURL,h}{\uvec{v}_h - \uvec{z}_h}
    \lesssim \tnorm{\DIV,h}{\uCh \uvec{v}_h},
  \] 
  with hidden constant only depending on $\Omega$, the mesh regularity parameter, and $k$.
\end{theorem}

\begin{proof}
  See Section \ref{sec:uCh.poincare}.
\end{proof}

\begin{theorem}[Poincar\'e inequality for the divergence]\label{thm:Dh.poincare}
  For all $\uvec{w}_h \in \Xdiv{h}$, it holds
  \[ 
    \inf_{\uvec{z}_h \in \Ker \Dh} \tnorm{\Dh,h}{\uvec{w}_h - \uvec{z}_h}
    \lesssim \norm{L^2(\Omega)}{\Dh \uvec{w}_h},
  \] 
  with hidden constant only depending on $\Omega$, the mesh regularity parameter, and $k$.
\end{theorem}

\begin{proof}
  See Section \ref{sec:Dh.poincare}.
\end{proof}


\section{Mimetic Poincar\'e inequalities} \label{sec:proof.poincaré}

  This section contains Poincar\'e inequalities in mimetic spaces that are instrumental in proving the main results stated in the previous section.
  Their proofs rely on the use of a tetrahedral submesh 
  $\fMh = \fTh \cup \fFh \cup \fEh \cup \fVh$ in the sense of \cite[Definition 1.8]{Di-Pietro.Droniou:20}, 
  with $\fTh$ collecting the tetrahedral subelements and $\fFh$, $\fEh$, and $\fVh$ 
  their faces, edges, and vertices, respectively.
We assume, for the sake of simplicity, that this submesh can be obtained adding as new vertices only centers of the faces and elements of $\Mh$.
As a result of the assumptions in \cite[Definition 1.8]{Di-Pietro.Droniou:20}, the regularity parameter of the submesh only depends on that of $\Mh$ and, for a given element $T \in \Th$, the diameters of the submesh entities contained in $\overline{T}$ are comparable to $h_T$ uniformly in $h$.

\subsection{Mimetic Poincar\'e inequality for collections of vertex values}

\begin{theorem}[Mimetic Poincar\'e inequality for collections of vertex values] \label{thm:Whitney.V}
  Let $(\alpha_V)_{V\in\Vh}\in\Real^{\Vh}$ be a collection of values at vertices.
  Then, there is  $C \in \Real$ such that 
  \begin{equation}\label{eq:Whitney.V:poincare}
    \sum_{T\in\Th} h_T^3 \sum_{V\in\VT} (\alpha_V - C)^2 
    \lesssim \sum_{T\in\Th} h_T \sum_{E\in\ET} \vert \jump{E}{\alpha_V} \vert^2,
  \end{equation}
    with hidden constant only depending on $\Omega$ and the mesh regularity parameter.%
  
\end{theorem}

\begin{proof}
  We extend the collection $(\alpha_V)_{V\in\Vh}$ to $\fVh$ setting %
    the values at face/element centers equal to the value taken at an arbitrary vertex of the face/element in question.
  
  For any simplex $S \in \fTh$ and any vertex $V\in\VS$ (with $\VS$ collecting the vertices of $S$), let $\phi_{S,V}$ denote the restriction to $S$ of the piecewise affine ``hat'' function associated with $V$ given by \eqref{eq:S.DR.0}, and let $\phi_h \in H^1(\Omega)$ be the piecewise polynomial function defined by setting
  \begin{equation}\label{eq:Whitney.V:phi.h}
    (\phi_h)_{\vert S} \coloneq \sum_{V\in\VS} (\alpha_V-C) \phi_{S,V}
    \qquad\forall S \in \fTh,
  \end{equation}
  where $C\in\Real$ is chosen so that the zero-average condition $\int_\Omega \phi_h = 0$ is satisfied.
  We next prove the following norm equivalences:
  \begin{align} \label{eq:W0.P1}
    \norm{L^2(\Omega)}{\phi_h}^2
    &\simeq \sum_{T\in\Th} h_T^3 \sum_{V\in\VT} (\alpha_V - C)^2,
    \\ \label{eq:W0.P2}
    \norm{\bvec{L}^2(\Omega; \Real^3)}{\GRAD \phi_h}^2
    &\simeq \sum_{T\in\Th} h_T \sum_{E\in\ET} \vert \jump{E}{\alpha_V} \vert^2.
  \end{align}
  The conclusion follows from the above relations writing
  \[
  \sum_{T\in\Th} h_T^3 \sum_{V\in\VT} (\alpha_V - C)^2
  \overset{\eqref{eq:W0.P1}}\lesssim
  \norm{L^2(\Omega)}{\phi_h}^2
  \lesssim
  \norm{\bvec{L}^2(\Omega;\Real^3)}{\GRAD\phi_h}^2
  \overset{\eqref{eq:W0.P2}}\lesssim
  \sum_{T\in\Th} h_T \sum_{E\in\ET} \vert \jump{E}{\alpha_V} \vert^2,
  \]
  where the second inequality follows from the continuous Poincar\'e--Wirtinger inequality, which holds since $\int_\Omega \phi_h = 0$.
  \smallskip\\
  \underline{(i) \emph{Proof of \eqref{eq:W0.P1}.}}
    For any $T \in \Th$, by regularity of
     the submesh, we have $h_S \simeq h_T$ for all $S \in \ST$ (with $\ST$ collecting the subelements contained in $T$).
    It holds
  \begin{equation} \label{eq:W0.P4}
    \begin{aligned}
      \norm{L^2(\Omega)}{\phi_h}^2
      \overset{\eqref{eq:Whitney.V:phi.h}}
      &=
      \sum_{T\in\Th} \sum_{S\in\ST} \int_S \bigg(\sum_{V\in\VS} (\alpha_V-C) \phi_{S,V}\bigg)^2
      \\
      &\simeq \sum_{T\in\Th} \sum_{S\in\ST} \sum_{V\in\VS} (\alpha_V-C)^2 \norm{L^2(S)}{\phi_{S,V}}^2
      \\
      \overset{\eqref{eq:W.norm.0}}&\simeq \sum_{T\in\Th} h_T^3 \sum_{S\in\ST} \sum_{V\in\VS} (\alpha_V-C)^2
      \\
      &\simeq \sum_{T\in\Th} h_T^3 \sum_{V\in\VT} (\alpha_V - C)^2,
    \end{aligned}
  \end{equation}
  where the second equivalence follows from the fact that $\card(\VS) \lesssim 1$,
    in the third one we have additionally used $h_S \simeq h_T$ for all $T \in \Th$ and all $S \in \ST$,
  and the last one is justified by the choice we made at the beginning for $\alpha_V$, $V \in \fVh \setminus \Vh$.
  This readily gives \eqref{eq:W0.P1}.
  \smallskip\\
  \underline{(ii) \emph{Proof of \eqref{eq:W0.P2}.}}
  The key argument to obtain \eqref{eq:W0.P2} lies in the de Rham theorem:
  Let $(\bvec{\psi}_{S,E})_{E\in\ES}$ (with $\ES$ collecting the edges of $S$) be the basis for the edge N\'ed\'elec space given by \eqref{eq:S.DR.1}.
  Then, summing \eqref{eq:W.diff.0}, we have
  \begin{equation} \label{eq:W0.W1}
    \GRAD \left(\sum_{V\in\VS} \alpha_V \phi_{S,V} \right)
    = \sum_{E\in\ES} \jump{E}{\alpha_V} \bvec{\psi}_{S,E} .
  \end{equation}
  Starting from \eqref{eq:W0.W1} and proceeding in a similar way as in \eqref{eq:W0.P4} with \eqref{eq:W.norm.1} replacing \eqref{eq:W.norm.0}, we have 
  \begin{equation} \label{eq:W0.P5}
    \int_{\Omega} \Vert\GRAD\phi_h\Vert^2 \simeq 
    \sum_{T\in\Th} h_T \sum_{S\in\ST} \sum_{E\in\ES} \vert \jump{E}{\alpha_V} \vert^2,
  \end{equation}
    with, for any $\bvec{v} \in \Real^3,$ $\norm{}{\bvec{v}}$ denoting the Euclidian norm of $\bvec{v}$.
  Now, for any edge $E \in \ES$ of any simplex $S \in \ST$, 
  either $E \in \ET$ %
    or, by the choice made at the beginning of this proof for $\alpha_V$ with $V$ face or element center, $\jump{E}{\alpha_V}$ can be computed as the sum of jumps along the boundary of $T$ (i.e. $\jump{E}{\alpha_V} =  \sum_{E'\in\ET} \omega_{E'E} \jump{E'}{\alpha_V}$ for $\omega_{E'E} \in \lbrace -1, 0, 1 \rbrace$).
  Therefore, 
  \[
  \sum_{S\in\ST} \sum_{E\in\ES} \vert \jump{E}{\alpha_V} \vert^2 \simeq 
  \sum_{E\in\ET} \vert \jump{E}{\alpha_V} \vert^2,
  \]
  and we infer \eqref{eq:W0.P2} from \eqref{eq:W0.P5}.
\end{proof}

\subsection{Mimetic Poincar\'e inequality for collections of edge values}

If the topology of the domain is non-trivial, a suitable condition for each void must be satisfied in order to establish a mimetic Poincar\'e inequality for collections of edge values.
Denote by $b_2$ the second Betti number, i.e., the number of voids encapsulated by $\Omega$.
Let $(\faces{\gamma_i})_{1\leq i \leq b_2}$ denote 
  collections of boundary faces such that $\gamma_i \coloneq \bigcup_{F \in \faces{\gamma_i}} \overline{F}$ is the boundary if the $i$th void.
  We start by proving a necessary and sufficient condition under which a function in the lowest-order Raviart--Thomas--N\'ed\'elec face space on the tetrahedral submesh is the curl of a function in the edge N\'ed\'elec space on the same mesh.

\begin{lemma}[Condition on the cohomology]\label{lem:condition.cohom.div}
    Denote by $\RT{1}(\fTh)$ and $\NE{1}(\fTh)$ lowest-order face and edge finite element spaces on the submesh.
    Then, for all $\bvec{\phi}_h \in \RT{1}(\fTh)$, there exists $\bvec{\chi}_h \in \NE{1}(\fTh)$ such that $\bvec{\phi}_h = \CURL \bvec{\chi}_h$
  if and only if
  \begin{equation}\label{eq:im.curl.NE1}
    \text{%
      $\DIV \bvec{\phi}_h = 0$ and
      $\sum_{F\in\faces{\gamma_i}} \int_F \bvec{\phi}_h \cdot \normal_\Omega = 0$
      for all integer $i$ such that $1 \leq i \leq b_2$,
    }
  \end{equation}
    with $\normal_\Omega$ denoting the unit normal vector to the boundary of $\Omega$ pointing out of $\Omega$.
\end{lemma}

\begin{proof}
    To check that \eqref{eq:im.curl.NE1} is necessary, notice that the first condition comes from the identity $\DIV\CURL = 0$,  while the second one follows from Green's theorem along with the fact that the flux of the curl of any function across a closed boundary is zero.

  We next prove that condition \eqref{eq:im.curl.NE1} is sufficient using a counting argument.
  From the de Rham Theorem, we know that that the dimension of the space of harmonic forms is precisely the number of voids $b_2$.
  For all integer $i$ such that $1 \leq i \leq b_2$, we define the linear form $L_i$ such that, for any vector-valued function $\bvec{\phi}$ smooth enough,
  \[
    L_i(\bvec{\phi}) \coloneq -\sum_{F\in\faces{\gamma_i}} \int_F \bvec{\phi} \cdot \normal_\Omega .
  \]
  For any $1 \leq i \leq b_2$, let $\bvec{x}_i \in \Real^3$ be any point inside the $i$th void and consider the function $\bvec{\phi}_i: \Real^3 \ni \bvec{x}\mapsto \frac{\bvec{x}-\bvec{x}_i}{\Vert \bvec{x} - \bvec{x}_i \Vert^3} \in \Real^3$.
  Noticing that $\DIV \bvec{\phi}_i = 0$ and applying the divergence theorem inside the $j$th void, 
  we infer that $L_j(\bvec{\phi}_i) = 0$ if $j \neq i$.
    Let us now show that $L_i(\bvec{\phi}_i) > 0$.
    Let $r > 0$ be the distance between $\bvec{x}_i$ and $\Omega$.
    Denoting by $\sphere_i$ the sphere of radius $\frac{r}2$ centred in $\bvec{x}_i$,
  we have that $\int_{\sphere_i} \bvec{\phi}_i \cdot \frac{(\bvec{x} - \bvec{x}_i)}{\Vert \bvec{x}- \bvec{x}_i \Vert} = \int_{\sphere_i} 4 r^{-2} = 4 \pi > 0$.
  Applying once again the divergence theorem to $\bvec{\phi}_i$ on the 
  volume $\mathcal{V}_i$ enclosed between $\sphere_i$ and $\gamma_i$, we have that
  $0 = \int_{\mathcal{V}_i} \DIV \bvec{\phi}_i = L_i(\bvec{\phi}_i) - \int_{\sphere_i} \bvec{\phi}_i \cdot \frac{(\bvec{x} - \bvec{x}_i)}{\Vert \bvec{x}- \bvec{x}_i \Vert}$. Therefore, 
  \begin{equation}\label{eq:Li.phii>0}
    L_i(\bvec{\phi}_i) = \int_{\sphere_i} \bvec{\phi}_i \cdot \frac{(\bvec{x} - \bvec{x}_i)}{\Vert \bvec{x}- \bvec{x}_i \Vert} > 0.
  \end{equation}
  Denoting by $\RTproj{1}{h}$ the canonical interpolator onto $\cvec{RT}^1(\fTh)$, we know that, for any function $\bvec{\phi}$, $\DIV (\RTproj{1}{h} \bvec{\phi}) = \lproj{h}{0} (\DIV \bvec{\phi})$ and $L_i(\RTproj{1}{h} \bvec{\phi}) = L_i(\bvec{\phi})$ by definition of the interpolator.
  Therefore, for any integer $i$ such that $1 \leq i \leq b_2$, $\RTproj{1}{h} \bvec{\phi}_i$ is a discrete harmonic form and, by a counting argument, the linearly independent family $(\RTproj{1}{h} \bvec{\phi}_i)_{1 \leq i \leq b_2}$ spans the space of discrete harmonic forms.

  Let now $\bvec{\bvec{\phi}}_h$ be such that $\DIV \bvec{\bvec{\phi}}_h = 0$.
  Then, $\bvec{\bvec{\phi}}_h = \CURL \bvec{\chi}_h + \sum_{i = 1}^{b_2} \lambda_i\, \RTproj{1}{h} \bvec{\phi}_i$ for some $\bvec{\chi}_h \in \NE{1}(\fTh)$ and $(\lambda_i)_{1 \leq i \leq b_2}\in \Real^{b_2}$.
  We prove that the condition $L_i(\bvec{\phi}) = 0$ for all $1\le i\le b_2$ is sufficient to ensure that $\bvec{\bvec{\phi}}_h$ is in the range of $\CURL$ by contradiction.
  As a matter of fact, if this were not the case, then there would be $i_0$ such that $\lambda_{i_0} \neq 0$.
  However, by \eqref{eq:Li.phii>0}, this would also imply $L_{i_0}(\bvec{\bvec{\phi}}_h) = \lambda_{i_0} L_{i_0}(\RTproj{1}{h} \bvec{\phi}_{i0}) \neq 0$, which is the sought contradiction.
\end{proof}

\begin{theorem}[Mimetic Poincar\'e inequality for collections of edge values] \label{thm:Whitney.E}
  Let $(\alpha_F)_{F\in\Fh}\in\Real^{\Fh}$ be a collection of values at faces satisfying
  \begin{gather}\label{eq:Whitney.E:condition.1}
    \text{%
      $\sum_{F\in\FT} \omega_{TF} \alpha_F = 0$ for all $T \in \Th$ and
    }
    \\ \nonumber 
    \text{%
      $\sum_{F\in\faces{\gamma_i}} \wOF \alpha_F = 0$
      for all integer $i$ such that $1\le i\le b_2$,
    }
  \end{gather}
  where $\wOF \in \lbrace -1, 1 \rbrace$ is such that $\wOF \nF$ points outside the domain $\Omega$.
  Then, there is a collection $(\alpha_E)_{E\in\Eh}\in\Real^{\Eh}$ of values at edges such that, for all $F\in\Fh$, 
  \begin{equation}\label{eq:Whitney.E:poincare}
    \text{%
      $\sum_{E\in\EF}\omega_{FE}\alpha_E = \alpha_F$
      and $\sum_{T\in\Th} h_T\sum_{E\in\ET} \alpha_E^2
      \lesssim \sum_{T\in\Th} h_T^{-1} \sum_{F\in\FT} \alpha_F^2$
    }
  \end{equation}
  with hidden constant only depending on $\Omega$ and the mesh regularity parameter.
\end{theorem}

\begin{proof}
  Let $(\bvec{\phi}_{E,S})_{S\in\fTh,E\in\ES}$ and $(\bvec{\psi}_{F,S})_{S\in\fTh,F\in\FS}$ (with $\FS$ denoting the set of triangular faces of $S$) be the families of basis functions respectively given by \eqref{eq:S.DR.1} and \eqref{eq:S.DR.2} below.
  The main difficulty is to extend the family $(\alpha_F)_{F\in\Fh}$ to a family $(\alpha_F)_{F\in \fFh}$ satisfying, for all $S\in\fTh$, 
  \[
  \sum_{F\in\FS} \omega_{SF} \alpha_F = 0.
  \]
  We perform the construction locally on each element $T\in\Th$.
  Let $g \in L^2(T)$ be the piecewise constant function on $\ST$ such that
  \begin{equation}\label{eq:Whitney.E:g}
    g_{|S} \coloneq - \frac{1}{\vert S \vert}\sum_{F\in\FS\cap\FT} \wTF \alpha_F
    \qquad\forall S \in \ST.
  \end{equation}
  Recalling \eqref{eq:Whitney.E:condition.1}, by definition we have $\int_T g = 0$.
  Hence, using Lion's Lemma \cite[Theorem 3.1.e]{Amrouche.Ciarlet.ea:15}, we infer the existence of $\bvec{u} \in \bvec{H}^1_0(T;\Real^3)$ such that
  \begin{equation}\label{eq:lions.u}
    \text{%
      $\DIV \bvec{u} = g$ and
      $%
        \seminorm{\bvec{H}^1(T;\Real^3)}{\bvec{u}}
      \lesssim
      \norm{L^2(T)}{g} \lesssim h_T^{-\frac32} \sum_{F\in\FT} \vert \alpha_F \vert$,
    }
  \end{equation}
  where the last inequality follows from the definition of $g$ after observing that $|S| \simeq h_T^3$ for all $S \in \ST$ by mesh regularity.
  Setting $\alpha_F \coloneq \int_F \bvec{u} \cdot \nF$ (with $\nF$ denoting the unit normal vector to $F$, with orientation consistent with that of $F$), 
  for all $S \in \ST$ and all $F \in \FS \setminus \FT$.
  We infer that, for all $S \subset \fTh$, $\sum_{F\in\FS} \omega_{SF} \alpha_F = 0$, noticing that 
  \begin{equation*}
    \begin{aligned}
      \sum_{F\in\FS} \omega_{SF} \alpha_F
      &= \sum_{F\in \FS \setminus \FT} \omega_{SF} \alpha_F + \sum_{F\in \FS \cap \FT} \omega_{SF} \alpha_F\\
      &= \sum_{F\in \FS} \omega_{SF} \int_F \bvec{u} \cdot \nF + \sum_{F\in \FS \cap \FT} \omega_{SF} \alpha_F\\
      &= \int_S \DIV \bvec{u} + \sum_{F\in \FS \cap \FT} \omega_{SF} \alpha_F\\
      \overset{\eqref{eq:lions.u},\,\eqref{eq:Whitney.E:g}}&= -\sum_{F\in \FS \cap \FT} \omega_{SF} \alpha_F + \sum_{F\in \FS \cap \FT} \omega_{SF} \alpha_F = 0,
    \end{aligned}
  \end{equation*}
  where we have used the fact that $\bvec{u} \cdot \nF = 0$ on every face $F \in \FT$ lying on the boundary of $T$ to obtain the second equality.
  Let $\overline{\bvec{u}}_T \coloneq \frac{1}{|T|}\int_T \bvec{u}$ and, for all integer $i$ such that $1\le i\le 3$, denote by $u_i$ and $\overline{u}_i$ the $i$th components of $\bvec{u}$ and $\overline{\bvec{u}}_T$, respectively.
  It holds
  \[
  \begin{aligned}
    |T|\,\overline{u}_i
    = \int_T u_i 
    =
    \int_T \bvec{u}\cdot \bvec{e}_i
    &= \int_T \bvec{u}\cdot \GRAD(x_i - \bvec{x}_T \cdot \bvec{e}_i)
    \\
    &= - \int_T \DIV \bvec{u} ~ (x_i - \bvec{x}_T \cdot \bvec{e}_i)
    = - \int_T g ~ (x_i - \bvec{x}_T \cdot \bvec{e}_i) 
    \lesssim h_T \sum_{F\in\FT} \vert \alpha_F \vert,
  \end{aligned}  
  \]
  so that, recalling that $|T| \simeq h_T^3$,
  \begin{equation}\label{eq:norm.bar.u}
    \norm{\bvec{L}^2(T;\Real^3)}{\overline{\bvec{u}}_T}
    \lesssim h_T^{-\frac12} \sum_{F\in\FT} \vert \alpha_F \vert.
  \end{equation}
  Therefore, we can use the Poincaré inequality on the domain $T$ to write
  \[
  \begin{aligned}
    \norm{\bvec{L}^2(T;\Real^3)}{\bvec{u}}
    &\leq \norm{\bvec{L}^2(T;\Real^3)}{\bvec{u} - \overline{\bvec{u}}_T} 
    + \norm{\bvec{L}^2(T;\Real^3)}{\overline{\bvec{u}}_T}
    \\
    &\lesssim h_T 
      \seminorm{\bvec{H}^1(T;\Real^3)}{\bvec{u}}
    + \norm{\bvec{L}^2(T;\Real^3)}{\overline{\bvec{u}}_T}
    \overset{\eqref{eq:lions.u},\,\eqref{eq:norm.bar.u}}\lesssim
    h_T^{-\frac12} \sum_{F\in\FT} \vert \alpha_F \vert.
  \end{aligned}
  \]
  Combining this result with the continuous trace inequality we have, for all $S \in \ST$
  and all $F \in \FS \setminus \FT$,
  \[
  \vert \alpha_F \vert
  \lesssim h_F \norm{\bvec{L}^2(F;\Real^3)}{\bvec{u}} 
  \lesssim h_T^{\frac12} \norm{\bvec{L}^2(T;\Real^3)}{\bvec{u}}
  + h_T^{\frac32} %
    \seminorm{\bvec{H}^1(T; \Real^3)}{\bvec{u}}
  \lesssim \sum_{F\in\FT} \vert \alpha_F \vert .
  \]
  Therefore, summing over all tetrahedra inside $T$ and all tetrahedral faces, we obtain
  \begin{equation}\label{eq:PW1.alpha}
    \sum_{S\in\ST} \sum_{F\in\FS} \alpha_F^2 \lesssim \sum_{F\in\FT} \alpha_F^2 .
  \end{equation}

  We next define the following piecewise polynomial function:
  \begin{equation} \label{eq:W1.P0}
    \bvec{\psi}_h \coloneq \sum_{S\in\fTh} \sum_{F\in\FS} \alpha_F \bvec{\psi}_{F,S} \in\Hdiv{\Omega}.
  \end{equation}
  Since $\DIV \bvec{\psi}_h = 0$ and, for all $1 \leq i \leq b_2$, $\sum_{F\in\faces{\gamma_i}} \wOF \int_F \bvec{\psi}_h \cdot \nF = 0$, we can use Lemma \ref{lem:condition.cohom.div} to
  infer from the uniform Poincar\'e inequality on the simplicial de Rham complex \cite{Arnold:18}
  the existence of 
  \[
  \bvec{\phi}_h \coloneq \sum_{S\in\fTh}\sum_{E\in\ES} \alpha_E \bvec{\phi}_{E,S} \in \Hcurl{\Omega}
  \]
  such that 
  \begin{equation}\label{eq:curl.surjectivity}
    \text{%
      $\CURL \bvec{\phi}_h = \bvec{\psi}_h$
      and
      $\norm{\bvec{L}^2(\Omega;\Real^3)}{\bvec{\phi}_h}
      \lesssim \norm{\bvec{L}^2(\Omega;\Real^3)}{\bvec{\psi}_h}$.
    }
  \end{equation}
  Summing \eqref{eq:W.diff.1}, we have
  \begin{equation}\label{eq:W1.P1}
    \CURL \bvec{\phi}_h = \sum_{S\in\fTh}\sum_{F\in\FS} \left( \sum_{E\in\EF} \wFE \alpha_E \right) \bvec{\psi}_{F,S}.
  \end{equation}
  Hence, equating \eqref{eq:W1.P0} and \eqref{eq:W1.P1}, we infer that, for all $F\in\Fh\subset\fFh$, $\sum_{E\in\EF}\omega_{FE}\alpha_E = \alpha_F$.
  Moreover, noticing that both $\bvec{\phi}_{E,S}$ and $\bvec{\psi}_{F,S}$ are only supported in $S$, we have 
  \begin{equation}\label{eq:PW1.norm.E}
    \begin{aligned}
      \norm{\bvec{L}^2(\Omega;\Real^3)}{\bvec{\phi}_h}^2 
      &= \sum_{T\in\Th}\sum_{S\in\ST} \int_S \left( \sum_{E\in\ES} \alpha_E \bvec{\phi}_{E,S}\right)^2
      \\
      &\simeq \sum_{T\in\Th}\sum_{S\in\ST} \sum_{E\in\ES} \alpha_E^2 \norm{\bvec{L}^2(S;\Real^3)}{\bvec{\phi}_{E,S}}^2
      \\
      \overset{\eqref{eq:W.norm.1}}&\simeq \sum_{T\in\Th}\sum_{S\in\ST} h_S \sum_{E\in\ES} \alpha_E^2
      \\
      &\gtrsim \sum_{T\in\Th} h_T \sum_{E\in\ET} \alpha_E^2
    \end{aligned}
  \end{equation}
  and
  \begin{equation}\label{eq:PW1.norm.F}
    \begin{aligned}
      \norm{\bvec{L}^2(\Omega;\Real^3)}{\bvec{\psi}_h}^2 
      &= \sum_{T\in\Th}\sum_{S\in\ST} \int_S \left( \sum_{F\in\FS} \alpha_F \bvec{\psi}_{F,S}\right)^2
      \\
      &\simeq \sum_{T\in\Th}\sum_{S\in\ST} \sum_{F\in\FS} \alpha_F^2 \norm{\bvec{L}^2(S;\Real^3)}{\bvec{\psi}_{F,S}}^2
      \\
      \overset{\eqref{eq:W.norm.2}}&\simeq \sum_{T\in\Th}\sum_{S\in\ST} h_S^{-1} \sum_{F\in\FS} \alpha_F^2
      \\
      \overset{\eqref{eq:PW1.alpha}}&\lesssim \sum_{T\in\Th} h_T^{-1} \sum_{F\in\FT} \alpha_F^2,
    \end{aligned}
  \end{equation}
  where, in the last passage, we have additionally used the fact that $h_S^{-1} \lesssim h_T^{-1}$ for all $S \in \ST$ by mesh regularity.
  Finally, to prove \eqref{eq:Whitney.E:poincare}, it suffices to write
  \[
  \sum_{T\in\Th} h_T \sum_{E\in\ET} \alpha_E^2
  \overset{\eqref{eq:PW1.norm.E}}\lesssim \norm{\bvec{L}^2(\Omega;\Real^3)}{\bvec{\phi}_h}^2
  \overset{\eqref{eq:curl.surjectivity}}\lesssim
  \norm{\bvec{L}^2(\Omega;\Real^3)}{\bvec{\psi}_h}^2
  \overset{\eqref{eq:PW1.norm.F}}\simeq \sum_{T\in\Th} h_T^{-1} \sum_{F\in\FT} \alpha_F^2 .\qedhere
  \]
\end{proof}

\subsection{Mimetic Poincar\'e inequality for collections of face values}

\begin{theorem}[Mimetic Poincar\'e inequality for collections of face values] \label{thm:Whitney.F}
  Let $(\alpha_T)_{T\in\Th}\in\Real^{\Th}$ be a collection of values at elements. 
  Then, there is a collection $(\alpha_F)_{F\in\Fh}\in\Real^{\Fh}$ of values at faces such that, for all $T\in\Th$, 
  \begin{equation}\label{eq:Whitney.F:poincare}
    \text{%
      $\sum_{F\in\FT}\omega_{TF} \alpha_F = \alpha_T$
      and $\sum_{T\in\Th} h_T^{-1} \sum_{F\in\FT} \alpha_F^2
      \lesssim \sum_{T\in\Th} h_T^{-3} \alpha_T^2$
    }
  \end{equation}
  with hidden constant only depending on $\Omega$ and the mesh regularity parameter.
\end{theorem}

\begin{proof}
  Let $(\bvec{\phi}_F)_{F\in\fFh}$ and $(\psi_S)_{S\in\fTh}$ be the basis functions of the face Raviart--Thomas--N\'ed\'elec and of the fully discontinuous piecewise affine spaces on the tetrahedral submesh, respectively given by \eqref{eq:S.DR.2} and \eqref{eq:S.DR.3} below.
  Define the following piecewise polynomial function:
  \begin{equation}\label{eq:W2.P0}
    \psi_h \coloneq \sum_{T\in\Th}\sum_{S\in\ST} \alpha_T \frac{\vert S \vert}{\vert T \vert} \psi_S .
  \end{equation}  
  We infer from the uniform Poincar\'e inequality on the simplicial de Rham complex \cite{Arnold:18}
  the existence of 
  $\bvec{\phi}_h \coloneq \sum_{F\in\fFh} \alpha_F \bvec{\phi}_F\in \RT{1}(\fTh) \subset \Hdiv{\Omega}$ such that 
  \begin{equation}\label{eq:div.surjectivity}
    \text{%
      $\DIV \bvec{\phi}_h = \psi_h$
      and
      $\norm{\bvec{L}^2(\Omega;\Real^3)}{\bvec{\phi}_h}
      \lesssim \norm{L^2(\Omega)}{\psi_h}$.
    }
  \end{equation}
  Summing \eqref{eq:W.diff.2}, on the other hand, we obtain
  \begin{equation} \label{eq:W2.P3} 
    \DIV \bvec{\phi}_h = \sum_{T\in\Th}\sum_{S\in\ST} \sum_{F\in\FS}\omega_{SF} \alpha_F \psi_S,
  \end{equation}
  where $\omega_{SF}$ is the orientation of $F$ relative to $S$.
  Since each $\psi_S$ is supported in $S$, we infer equating \eqref{eq:W2.P0} and \eqref{eq:W2.P3} that, for all $T \in \Th$ and all $S \in \ST$,
  \[
  \alpha_T \frac{\vert S \vert}{\vert T \vert} = \sum_{F\in\FS}\omega_{SF} \alpha_F.
  \]
  For all $T\in\Th$, summing this relation over $S \in \ST$, we have
  \[
  \sum_{S\in\ST} \alpha_T \frac{\vert S \vert}{\vert T \vert}
  = \sum_{S\in\ST} \sum_{F\in\FS}\omega_{SF} \alpha_F
  \implies
  \alpha_T \cancelto{1}{\frac{\sum_{S\in\ST} \vert S \vert}{\vert T \vert}}
  = \sum_{F\in\FT}\omega_{TF} \alpha_F,
  \]
  where we have used the fact that the contributions from the simplicial faces internal to $T$ cancel out in the right-hand side.

  It remains to check that \eqref{eq:Whitney.F:poincare} holds.
  Noticing that the only $\bvec{\phi}_F$ supported in $S$ are those associated with its faces $F$ collected in the set $\FS$, we have
  \begin{equation}\label{eq:est.norm.phih}
    \begin{aligned}
      \norm{\bvec{L}^2(\Omega;\Real^3)}{\bvec{\phi}_h}^2 
      &= \sum_{T\in\Th} \sum_{S\in\ST} \int_S \bigg( \sum_{F\in\FS} \alpha_F \bvec{\phi}_F\bigg)^2
      \\
      &\simeq \sum_{T\in\Th} \sum_{S\in\ST} \sum_{F\in\FS} \alpha_F^2 \norm{\bvec{L}^2(S;\Real^3)}{\bvec{\phi}_F}^2
      \\
      \overset{\eqref{eq:W.norm.2}}&\simeq \sum_{T\in\Th} h_T^{-1} \sum_{S\in\ST} \sum_{F\in\FS} \alpha_F^2
      \\ 
      &\gtrsim \sum_{T\in\Th} h_T^{-1} \sum_{F\in\FT} \alpha_F^2,
    \end{aligned}
  \end{equation}
  where we have used $h_S^{-1} \simeq h_T^{-1}$ (consequence of mesh regularity) in the third line.
  Likewise, we have
  \begin{equation}\label{eq:est.norm.psih}
    \norm{L^2(\Omega)}{\psi_h}^2 
    = \sum_{T\in\Th} \sum_{S\in\ST} \alpha_T^2 \frac{\vert S \vert^2}{\vert T \vert^2} \norm{L^2(S)}{\psi_S}^2
    \overset{\eqref{eq:W.norm.3}}\simeq \sum_{T\in\Th} h_T^{-3} \alpha_T^2,
  \end{equation}
  where we have used $h_S \simeq h_T$ and $\vert S \vert \simeq \vert T \vert$.
  Finally, to prove \eqref{eq:Whitney.F:poincare}, we write
  \[
  \sum_{T\in\Th} h_T^{-1} \sum_{F\in\FT} \alpha_F^2
  \overset{\eqref{eq:est.norm.phih}}\lesssim \norm{\bvec{L}^2(\Omega;\Real^3)}{\bvec{\phi}_h}^2
  \overset{\eqref{eq:div.surjectivity}}\lesssim
  \norm{L^2(\Omega)}{\psi_h}^2
  \overset{\eqref{eq:est.norm.psih}}\simeq \sum_{T\in\Th} h_T^{-3} \alpha_T^2 .\qedhere
  \]
\end{proof}


\section{Proofs of Poincar\'e inequalities in DDR spaces} \label{sec:proof.DDR.poincare}

\subsection{Poincar\'e inequality for the gradient}\label{sec:uGh.poincare}

We start with the following preliminary lemma.

\begin{lemma}[Continuous inverse of the discrete gradient]\label{lem:uGh.inverse}
  For all $\underline{p}_h \in \Xgrad{h}$, there is $\underline{q}_h \in \Xgrad{h}$ such that
  \begin{equation}\label{eq:uGh.inverse}
    \text{%
      $\uGh\underline{q}_h = \uGh\underline{p}_h$
      and $\tnorm{\GRAD,h}{\underline{q}_h} \lesssim \tnorm{\CURL,h}{\uGh\underline{p}_h}$.
    }
  \end{equation}
\end{lemma}

\begin{proof}
  We provide an explicit definition of $\underline{q}_h$ and check that \eqref{eq:uGh.inverse} holds.
  Specifically, we let $\underline{q}_h \in \Xgrad{h}$ be such that, for all $E\in\Eh$,
  \begin{gather}\label{eq:qV}
    \jump{E}{q_V} = \int_E \GE \underline{p}_E \qquad \forall E\in\Eh,
    \\ \label{eq:qE}
    \int_E q_E\, r' = - \int_E \GE \underline{p}_E\, r
    + \jump{E}{q_V\, r}
    \qquad \forall r \in \Poly[0]{k}(E),
  \end{gather}
  for all $F\in\Fh$,
  \begin{equation}\label{eq:qF}
    \int_F q_F\,\DIV_F \bvec{v}
    = -\int_F \GF \underline{p}_F \cdot \bvec{v}
    + \sum_{E\in\EF}\omega_{FE}\int_E q_E\,(\bvec{v}\cdot\normal_{FE})
    \qquad \forall \bvec{v} \in \cRoly{k}(F),
  \end{equation}
  and, for all $T\in\Th$,
  \begin{equation}\label{eq:qT}
    \int_T q_T\,\DIV\bvec{v}
    + \sum_{F\in\FT} \omega_{TF}\int_F q_F\,(\bvec{v}\cdot\normal_F)
    \qquad \forall \bvec{v} \in \cRoly{k}(T).
  \end{equation}
  Notice that
    the vertex values $(q_V)_{V \in \Vh}$ are only defined up to a global constant and that
  $q_F$ (resp., $q_T$) is well-defined by condition \eqref{eq:qF} (resp., \eqref{eq:qT}) since $\DIV_F:\cRoly{k}(F)\to\Poly{k-1}(F)$ (resp., $\DIV:\cRoly{k}(T)\to\Poly{k-1}(T)$) is an isomorphism.
  \medskip\\
  \underline{1. \emph{Equality of the discrete gradient.}}
  Let us first briefly show that
  \begin{equation}\label{eq:uGh.q=uGh.p}
    \uGh \underline{q}_h = \uGh \underline{p}_h.
  \end{equation}
  By \eqref{eq:qV} and \eqref{eq:qE} along with the definition \eqref{eq:GE} of $\GE$, it holds
  \begin{equation}\label{eq:GE.q=GE.p}
    \GE \underline{q}_E = \GE \underline{p}_E \qquad \forall E \in \Eh.
  \end{equation}
  Let now $F \in \Fh$.
  By \cite[Lemma 14]{Bonaldi.Di-Pietro.ea:23} for $(k,d) = (0,2)$, $\int_F \GF \underline{p}_F \cdot \VROT_F r = -\sum_{E\in\EF}\omega_{FE}\int_E \GE \underline{p}_E\, r$ for all $r \in \Poly[0]{k+1}(F)$, 
  so that, by \eqref{eq:GE.q=GE.p}, $\Rproj{k}{F} \GF \underline{q}_F = \Rproj{k}{F} \GF \underline{p}_F$.
  On the other hand, plugging the definition \eqref{eq:qF} of $q_F$ into \eqref{eq:GF}, we readily infer that $\cRproj{k}{F} \GF \underline{q}_F = \cRproj{k}{F} \GF \underline{p}_F$.
  The above relations imply
  \begin{equation}\label{eq:GF.q=GF.p}
    \GF \underline{q}_F = \GF \underline{p}_F
    \qquad \forall F\in\Fh.
  \end{equation}
  The equality of the components associated with an element $T \in \Th$ is proved in a similar way. First, using again \cite[Lemma 14]{Bonaldi.Di-Pietro.ea:23}, this time with $(k,d) = (0,3)$ (which corresponds to \cite[Proposition 1]{Di-Pietro.Droniou:23*1}), we infer that $\int_T \GT \underline{q}_T \cdot \CURL \bvec{v} = - \sum_{F\in\FT} \int_F \GF \underline{q}_F \cdot (\bvec{v} \times \normal_F)$ for all $\bvec{v} \in \cGoly{k+1}(T)$.
  Accounting for \eqref{eq:GF.q=GF.p}, this yields $\Rproj{k}{T} \GT \underline{q}_T = \Rproj{k}{T} \GT \underline{p}_T$.
  Then, plugging the definition \eqref{eq:qT} of $q_T$ into \eqref{eq:GT}, we get $\cRproj{k}{T} \GT \underline{q}_T = \cRproj{k}{T} \GT \underline{p}_T$.
  These equalities give
  \begin{equation}\label{eq:GT.q=GT.p}
    \GT \underline{q}_T = \GT \underline{p}_T
    \qquad \forall F\in\Th.
  \end{equation}
  Gathering \eqref{eq:GE.q=GE.p}, \eqref{eq:GF.q=GF.p}, and \eqref{eq:GT.q=GT.p}, and recalling the definition \eqref{eq:uGh} of the discrete gradient \eqref{eq:uGh.q=uGh.p} follows.
  \medskip\\
  \underline{2. \emph{Continuity.}}
  Using the fact that, for all $T \in \Th$, $h_Y \lesssim h_T$ for all $Y \in \FT \cup \ET$ and that the number of faces of each element and of edges of each face is $\lesssim 1$ by mesh regularity, we have
  \[
  \sum_{T\in\Th} h_T \sum_{F\in\FT} h_F \sum_{E\in\EF} h_E \sum_{V\in\VE} |q_V|^2
  \lesssim \sum_{T\in\Th} h_T^3 \sum_{V\in\VT} |q_V|^2
  \overset{\eqref{eq:Whitney.V:poincare}}\lesssim \sum_{T\in\Th} h_T \sum_{E\in\ET} |\jump{E}{q_V}|^2,
  \]
  where, to apply Theorem \ref{thm:Whitney.V}, we have used the fact that vertex values are defined up to a constant.
  Taking absolute values in \eqref{eq:qV} and using a Cauchy--Schwarz inequality in the right-hand side, on the other hand, we obtain $|\jump{E}{q_V}| \lesssim h_E^{\frac12} \norm{L^2(E)}{\GE \underline{p}_E} \le h_T^{\frac12} \norm{L^2(E)}{\GE \underline{p}_E}$ for all $T \in \Th$ such that $E \in \ET$.
  Plugging this bound into the previous expression, we obtain
  \begin{equation}\label{eq:inverse.uGh:estimate.V}
    \sum_{T\in\Th} h_T \sum_{F\in\FT} h_F \sum_{E\in\EF} h_E \sum_{E\in\VE} |q_V|^2
    \lesssim \sum_{T\in\Th} h_T^2 \sum_{E\in\ET} \norm{L^2(E)}{\GE \underline{p}_E}^2
    \lesssim \tnorm{\CURL,h}{\uGh \underline{p}_h}^2,
  \end{equation}
  where the last inequality is obtained recalling the definition \eqref{eq:tnorm.curl.h} of $\tnorm{\CURL,h}{{\cdot}}$ with $\uvec{v}_h = \uGh \underline{p}_h$ and using mesh regularity.

  Taking in \eqref{eq:qE} $r$ such that $r' = q_E$, using Cauchy--Schwarz and trace inequalities in the right-hand side, simplifying, and squaring, we obtain
  \[
  \norm{L^2(E)}{q_E}^2
  \lesssim h_E^2 \norm{L^2(E)}{\GE\underline{p}_E}^2
  + h_E \sum_{V\in\VE} |q_V|^2,
  \]
  so that, using the fact that $h_E \le h$ for all $E\in\Eh$ in the first term,
  \begin{equation}\label{eq:inverse.uGh:estimate.E}
    \begin{aligned}
      \sum_{T\in\Th} h_T \sum_{F\in\FT} h_F \sum_{E\in\EF} \norm{L^2(E)}{q_E}^2
      &\lesssim
      h^2 \sum_{T\in\Th} h_T \sum_{F\in\FT} h_F \sum_{E\in\EF} \norm{L^2(E)}{\GE\underline{p}_E}^2
      \\
      &\quad
      + \sum_{T\in\Th} h_T \sum_{F\in\FT} h_F \sum_{E\in\EF} h_E \sum_{V\in\VE} |q_V|^2
      \\
      \overset{\eqref{eq:tnorm.curl.h},\,\eqref{eq:inverse.uGh:estimate.V}}&\lesssim
      \tnorm{\CURL,h}{\uGh \underline{p}_h}^2,
    \end{aligned}
  \end{equation}
  where, in the last inequality, we have additionally used the fact that $h \le \diam(\Omega) \lesssim 1$.

  To estimate the face components, we first take in \eqref{eq:qF} $\bvec{v}$ such that $\DIV_F\bvec{v} = q_F$ (this is possible since $\DIV_F:\cRoly{k}(F)\to\Poly{k-1}(F)$ is an isomorphism by \cite[Corollary 7.3]{Arnold:18}) use Cauchy--Schwarz and trace inequalities in the right-hand side to obtain
  \[
  \norm{L^2(F)}{q_F}^2
  \lesssim h_F^2 \norm{\bvec{L}^2(F;\Real^2)}{\GF\underline{p}_F}^2
  + h_F \sum_{E\in\EF} \norm{L^2(E)}{q_E}^2,
  \]
  so that, using the fact that $h_F \le h$ for all $F\in\Fh$ in the first term,
  \begin{equation}\label{eq:inverse.uGh:estimate.F}
    \begin{aligned}
      \sum_{T\in\Th} h_T \sum_{F\in\FT} \norm{L^2(F)}{q_F}^2
      &\lesssim h^2 \sum_{T\in\Th} h_T \sum_{F\in\FT} \norm{\bvec{L}^2(F;\Real^2)}{\GF\underline{p}_F}^2
      \\
      &\quad
      + \sum_{T\in\Th} h_T \sum_{F\in\FT} h_F \sum_{E\in\EF} \norm{L^2(E)}{q_E}^2
      \overset{\eqref{eq:tnorm.curl.h},\,\eqref{eq:inverse.uGh:estimate.E}}\lesssim
      \tnorm{\CURL,h}{\uGh \underline{p}_h}^2,
    \end{aligned}
  \end{equation}
    where we have again used the fact that $h \lesssim 1$ for the first term.

  The estimate of the element components is entirely similar: taking in \eqref{eq:qT} $\bvec{v} \in \cRoly{k}(T)$ such that $\DIV \bvec{v} = q_T$ 
  (which is possible since $\DIV : \cRoly{k}(T) \to \Poly{k-1}(T)$ is an isomorphism again by \cite[Corollary 7.3]{Arnold:18}),
  we obtain $\norm{L^2(T)}{q_T}^2 \lesssim h_T^2 \norm{\bvec{L}^2(T;\Real^3)}{\GT \underline{p}_T}^2 + h_T \sum_{F\in\FT} \norm{L^2(F)}{q_F}^2$. 
  Summing over $T\in\Th$, using the fact that $h_T \le h\lesssim 1$ for all $T\in\Th$ for the first term and \eqref{eq:inverse.uGh:estimate.F} for the second, 
  and recalling the definition \eqref{eq:tnorm.curl.h} of $\tnorm{\CURL,h}{{\cdot}}$, we arrive at
  \begin{equation}\label{eq:inverse.uGh:estimate.T}
    \sum_{T\in\Th} \norm{L^2(T)}{q_T}^2
    \lesssim \tnorm{\CURL,h}{\uGh \underline{p}_h}^2.
  \end{equation}

  Summing \eqref{eq:inverse.uGh:estimate.V},
  \eqref{eq:inverse.uGh:estimate.E},
  \eqref{eq:inverse.uGh:estimate.F}, and
  \eqref{eq:inverse.uGh:estimate.T}, the continuity bound in \eqref{eq:uGh.inverse} follows.
\end{proof}

\begin{proof}[Proof of Theorem \ref{thm:uGh.poincare}]
  Let $\underline{q}_h \in \Xgrad{h}$ be given by Lemma \ref{lem:uGh.inverse}.
  Noticing that $\underline{p}_h - \underline{q}_h \in \Ker \uGh$, we write
  \[
  \inf_{\underline{r}_h \in \Ker \uGh} \tnorm{\GRAD,h}{\underline{p}_h - \underline{r}_h}
  \le \tnorm{\GRAD,h}{\underline{p}_h - (\underline{p}_h - \underline{q}_h)}
  = \tnorm{\GRAD,h}{\underline{q}_h}
  \overset{\eqref{eq:uGh.inverse}}\lesssim \tnorm{\CURL,h}{\uGh \underline{p}_h}.\qedhere
  \]
\end{proof}


\subsection{Poincar\'e inequality for the curl}\label{sec:uCh.poincare}

The proof of Theorem \ref{thm:uCh.poincare} is analogous to that of Theorem \ref{thm:uGh.poincare} provided we establish the following result.

\begin{lemma}[Continuous inverse of the discrete curl]
  For any $\uvec{v}_h \in \Xcurl{h}$, there is $\uvec{z}_h \in \Xcurl{h}$ such that
  \begin{equation}\label{eq:uCh.inverse}
    \text{%
      $\uCh \uvec{z}_h = \uCh \uvec{v}_h$ and
      $\tnorm{\CURL,h}{\uvec{z}_h} \lesssim \tnorm{\DIV,h}{\uCh \uvec{v}_h}$.
    }
  \end{equation}
\end{lemma}

\begin{proof}
  Let
  \begin{equation}\label{eq:inverse.uCh:alpha.F}
    \alpha_F \coloneq \int_F \CF \uvec{v}_F
    \qquad\forall F \in \Fh.
  \end{equation}
  Recalling that $\Dh \uCh \uvec{v}_h = 0$ and using the definition \eqref{eq:DT} of $\DT$ it holds, for all $T \in \Th$,
  \[
  0 = \int_T \DT \uCT \uvec{v}_T
  = \sum_{F \in \FT} \omega_{TF} \int_F \CF \uvec{v}_F
  = \sum_{F \in \FT} \omega_{TF} \alpha_F,
  \]
  showing that \eqref{eq:Whitney.E:condition.1} is met.
  If the domain encapsulates $b_2 > 0$ voids, we infer using \eqref{eq:CF} for $r = 1$, that, for all integers $1 \leq i \leq b_2$, 
  \[
  \sum_{F\in\faces{\gamma_i}} \wOF \alpha_F 
  = \sum_{F\in\faces{\gamma_i}} \wOF \int_F \CF \uvec{v}_F \\ 
  = -\sum_{F\in\faces{\gamma_i}} \wOF \sum_{E\in\EF} \wFE \int_E v_E = 0,
  \]
  where we have used the fact that each edge in the summation appears exactly twice with opposite coefficients, and therefore cancels out.
  We can thus invoke Theorem \ref{thm:Whitney.E} to infer the existence of a collection of values at edges $(\alpha_E) \in \Real^{\Eh}$ satisfying \eqref{eq:Whitney.E:poincare}.
  We then take $z_E \in \Poly{0}(E)$ such that
  \begin{equation}\label{eq:z.E}
    \int_E z_E = h_E\, z_E = \alpha_E
    \qquad \forall E \in \Eh.
  \end{equation}
  For all $F \in \Fh$, the face components $\bvec{z}_{\cvec{R},F} \in \Roly{k-1}(F)$ and $\bvec{z}_{\cvec{R},F}^{\compl} \in \cRoly{k}(F)$ are selected such that
  \begin{equation}\label{eq:z.RF}
    \int_F \bvec{z}_{\cvec{R},F}\cdot\VROT_F r
    = \int_F \CF \uvec{v}_F\, r
    + \sum_{E \in \EF} \int_E z_E\, r
    \qquad\forall r \in \Poly[0]{r}(F)
  \end{equation}
  and
  \begin{equation}\label{eq:z.cRF}
    \bvec{z}_{\cvec{R},F}^{\compl} = \bvec{0}.
  \end{equation}
  Similarly, for any $T \in \Th$, the element components $\bvec{z}_{\cvec{R},T} \in \Roly{k-1}(T)$ and $\bvec{z}_{\cvec{R},T}^{\compl} \in \cRoly{k}(T)$ satisfy
  \begin{equation}\label{eq:z.RT}
    \int_T \bvec{z}_{\cvec{R},T} \cdot \CURL \bvec{w}
    = \int_T \CT \uvec{v}_T \cdot \bvec{w}
    + \sum_{F \in \FT} \omega_{TF} \int_F \trFt \uvec{z}_F \cdot (\bvec{w} \times \normal_F)
    \qquad \forall \bvec{w} \in \cGoly{k}(T)
  \end{equation}
  and
  \begin{equation}\label{eq:z.cRT}
    \bvec{z}_{\cvec{R},T}^{\compl} = \bvec{0}.
  \end{equation}
  Notice that \eqref{eq:z.RT} defines $\bvec{z}_{\cvec{R},T}$ uniquely since $\CURL : \cGoly{k}(T) \mapsto \Roly{k-1}(T)$ is an isomorphism by \cite[Corollary 7.3]{Arnold:18}.
  \medskip\\
  \underline{1. \emph{Equality of the discrete curl.}}
  By definition of $z_E$, it holds $\lproj{0}{F} \CF \uvec{z}_F = \lproj{0}{F} \CF \uvec{v}_F$ for all $F \in \Fh$.
  The equality of the higher-order components is obtained plugging \eqref{eq:z.RF} into the definition \eqref{eq:CF} of $\CF \uvec{z}_F$, which leads to $\int_F \CF \uvec{z}_F\, r = \int_F \CF \uvec{v}_F\, r$ for all $r \in \Poly[0]{k}(F)$.
  By \cite[Proposition~4]{Di-Pietro.Droniou:23*1} (which corresponds to \cite[Lemma~14]{Bonaldi.Di-Pietro.ea:23} with $(d,k) = (3,1)$), the equality of the face curls implies $\Gproj{k}{T}(\CT \uvec{z}_T) = \Gproj{k}{T} (\CT \uvec{v}_T)$.
  The equality of the projections on $\cRoly{k}(T)$ results plugging \eqref{eq:z.RT} into the definition \eqref{eq:CT} of $\CT \uvec{z}_{T}$ with test function $\bvec{w} \in \cGoly{k}(T)$.
  Gathering the above results, we obtain $\CT \uvec{z}_T = \CT \uvec{v}_{T}$ for all $T \in \Th$, from which $\uCh \uvec{z}_h = \uCh \uvec{v}_h$ follows recalling the definition \eqref{eq:uCh} of the discrete curl.
  \medskip\\
  \underline{2. \emph{Continuity.}}
  Let us now show the bound in \eqref{eq:uCh.inverse}.
  Concerning edge components, we write
  \begin{equation}\label{eq:est.z.E}
    \begin{aligned}
      \sum_{T\in\Th} h_T \sum_{F\in\FT} h_F \sum_{E\in\EF} \norm{L^2(E)}{z_E}^2
      &\lesssim
      \sum_{T\in\Th} h_T \sum_{E\in\ET} \alpha_E^2
      \\
      &\lesssim
      \sum_{T\in\Th} h_T^{-1} \sum_{F\in\FT} \left(
      \int_F \CF \uvec{v}_F
      \right)^2
      \lesssim \tnorm{\DIV,h}{\uCh \uvec{v}_h}^2,
    \end{aligned}
  \end{equation}
  where we have used \eqref{eq:z.E} along with the fact that each edge $E \in \ET$ is shared by exactly two faces in $\FT$,
  the mimetic Poincar\'e inequality \eqref{eq:Whitney.E:poincare} together with the definition \eqref{eq:inverse.uCh:alpha.F} of $\alpha_F$ in the second passage,
  and $\left|\int_F \CF \uvec{v}_F\right| \lesssim |F|^{\frac12} \norm{L^2(F)}{\CF \uvec{v}_F} \lesssim h_F \norm{L^2(F)}{\CF \uvec{v}_F} \le h_T \norm{L^2(F)}{\CF \uvec{v}_F}$ for all $F \in \Fh$ and all $T \in \Th$ to which $F$ belongs together with the definition \eqref{eq:tnorm.div.h} of $\tnorm{\DIV,h}{{\cdot}}$ to conclude.
  
  To estimate the face component, we let $r$ in \eqref{eq:z.RF} be such that $\VROT_F r = \bvec{z}_{\cvec{R},F}$ and use Cauchy--Schwarz, inverse, and trace inequalities in the right-hand side to infer $\norm{\bvec{L}^2(F;\Real^2)}{\bvec{z}_{\cvec{R},F}} \lesssim h_F \norm{L^2(F)}{\CF \uvec{v}_F} + h_F^{\frac12}\sum_{E\in\EF} \norm{L^2(E)}{z_E}$.
  Squaring the above relation and using standard inequalities for the square of a finite sum of terms, we obtain, after noticing that $h_F \le h$ for all $F \in \Fh$,
  \begin{equation}\label{eq:est.z.RF}
    \begin{aligned}
      \sum_{T\in\Th} h_T \sum_{F\in\FT} \norm{\bvec{L}^2(F;\Real^2)}{\bvec{z}_{\cvec{R},F}}^2
      &\lesssim h^2 \sum_{T\in\Th} h_T \sum_{F\in\FT} \norm{L^2(F)}{\CF \uvec{v}_F}^2
      \\
      &\quad
      + \sum_{T\in\Th} h_T \sum_{F\in\FT} h_F \sum_{E\in\EF} \norm{L^2(E)}{z_E}^2
      \lesssim \tnorm{\DIV,h}{\uCh \uvec{v}_h}^2,
    \end{aligned}
  \end{equation}
  where the conclusion follows noticing that $h \le \diam(\Omega) \lesssim 1$, recalling the definition \eqref{eq:tnorm.div.h} of $\tnorm{\DIV,h}{{\cdot}}$ for the first term, and invoking \eqref{eq:est.z.E} for the second one.

  The estimate of the element component is obtained in a similar way, starting from \eqref{eq:z.RT} with $\bvec{w}$ such that $\CURL \bvec{w} = \bvec{z}_{\cvec{R},T}$, leading to
  \begin{equation}\label{eq:est.z.RT}
    \sum_{T\in\Th} \norm{\bvec{L}^2(T;\Real^3)}{\bvec{z}_{\cvec{R},T}}^2
    \lesssim \tnorm{\DIV,h}{\uCh \uvec{v}_h}^2.
  \end{equation}
  Summing \eqref{eq:est.z.E}, \eqref{eq:est.z.RF}, and \eqref{eq:est.z.RT}, recalling the definition \eqref{eq:tnorm.curl.h} of $\tnorm{\CURL,h}{\uvec{z}_h}$ as well as \eqref{eq:z.cRF} and \eqref{eq:z.cRT}, the bound in \eqref{eq:uCh.inverse} follows.
\end{proof}


\subsection{Poincar\'e inequality for the divergence}\label{sec:Dh.poincare}

Theorem \ref{thm:Dh.poincare} is established in the same way as Theorem \ref{thm:uGh.poincare} (see the end of Section \ref{sec:uGh.poincare}) starting from the following result.

\begin{lemma}[Continuous inverse of the discrete divergence]
  For any $\uvec{w}_h \in \Xdiv{h}$, there is $\uvec{z}_h \in \Xdiv{h}$ such that
  \begin{equation}\label{eq:Dh.inverse}
    \text{%
      $\Dh\uvec{z}_h = \Dh\uvec{w}_h$
      and $\tnorm{\DIV,h}{\uvec{z}_h} \lesssim \norm{L^2(\Omega)}{\Dh\uvec{w}_h}$.
    }
  \end{equation}
\end{lemma}

\begin{proof}
  The face components of $\uvec{z}_h$ are obtained applying Theorem \ref{thm:Whitney.F} with $\alpha_T = \int_T\DT\uvec{w}_T$ for all $T\in\Th$ and letting $z_F \in \Poly{0}(F)$ be such that
  \begin{equation}\label{eq:z.F}
    \int_F z_F = |F|\, z_F =\alpha_F
    \qquad \forall F\in\Fh.
  \end{equation}
  For all $T\in\Th$, the element component $\bvec{z}_{\cvec{G},T} \in \Goly{k-1}(T)$ is defined by the following relation:
  \begin{equation}\label{eq:z.GT}
    \int_T \bvec{z}_{\cvec{G},T}\cdot\GRAD q
    = - \int_T \DT\uvec{w}_T\,q
    + \sum_{F\in\FT}\omega_{TF}\int_F z_F\,q
    \qquad\forall q\in\Poly[0]{k}(T).
  \end{equation}
  Finally, we set
  \begin{equation}\label{eq:z.cGT}
    \bvec{z}_{\cvec{G},T}^{\compl} = \bvec{0}
    \qquad \forall T\in\Th.
  \end{equation}
  \\
  \underline{1. \emph{Equality of the discrete divergence.}}
  To check the first condition in \eqref{eq:Dh.inverse}, it suffices to show that $\DT\uvec{z}_T = \DT\uvec{w}_T$ for a generic $T\in\Th$.
  To this end, we start by noticing that
  \[
  \int_T \DT \uvec{w}_T
  = \alpha_T
  = \sum_{F\in\FT}\omega_{TF}\alpha_F
  = \sum_{F\in\FT}\omega_{TF}\int_F z_F
  = \int_T \DT \uvec{z}_T,
  \]
  showing that $\lproj{0}{T}( \DT \uvec{z}_T ) = \lproj{0}{T}( \DT \uvec{w}_T )$.
  To show the equivalence of the higher-order components, it suffices to use \eqref{eq:z.GT} in \eqref{eq:DT} written for $\uvec{z}_T$ to infer $\int_T \DT \uvec{z}_T\,q = \int_T \DT \uvec{w}_T\,q$ for all $q\in\Poly[0]{k}(T)$.
  \medskip\\
  \underline{2. \emph{Continuity.}}
  Let us now show the continuity bound in \eqref{eq:Dh.inverse}.
  Observing that, by \eqref{eq:z.F}, for all $F \in \Fh$ it holds $\norm{L^2(F)}{z_F}^2 = |F|\, z_F^2 = |F|^{-1} \alpha_F^2 \lesssim h_T^{-2} \alpha_F^2$ for all $T \in\Th$ such that $F \in \FT$ (the last inequality being a consequence of mesh regularity), we have
  \begin{equation}\label{eq:est.z.F}
    \begin{aligned}
    \sum_{T\in\Th} h_T \sum_{F\in\FT} \norm{L^2(F)}{z_F}^2
    &\lesssim \sum_{T\in\Th} h_T^{-1} \sum_{F\in\FT} \alpha_F^2
    \overset{\eqref{eq:Whitney.F:poincare}}\lesssim \sum_{T\in\Th} h_T^{-3} \alpha_T^2
    \\
    &\lesssim \sum_{T\in\Th} h_T^{-3}\, |T|\, \norm{L^2(T)}{\lproj{0}{T}(\DT\uvec{w}_T)}^2
    \lesssim \norm{L^2(\Omega)}{\Dh\uvec{w}_h}^2,
    \end{aligned}
  \end{equation}
  where, to pass to the second line, we have used the mesh regularity assumption to infer $h_T^{-3}\, |T| \lesssim 1$.
  To estimate the element components, we take $q$ in \eqref{eq:z.GT} such that $\GRAD q = \bvec{z}_{\cvec{G},T}$, use Cauchy--Schwarz, trace, and inverse inequalities in the right-hand side, pass to the square and use standard inequalities for the square of a finite sum of terms to obtain
  \[
  \norm{\bvec{L}^2(T;\Real^3)}{\bvec{z}_{\cvec{G},T}}^2
  \lesssim h_T^2 \norm{L^2(T)}{\DT\uvec{w}_T}^2
  + h_T \sum_{F\in\FT} \norm{L^2(F)}{z_F}^2.
  \]
  Summing the above relation over $T\in\Th$, using the fact that $h_T \le \diam(\Omega) \lesssim 1$ for the first term in the right-hand side and \eqref{eq:est.z.F} for the second, we obtain
  \begin{equation}\label{eq:est.z.GT}
    \sum_{T\in\Th} \norm{\bvec{L}^2(T;\Real^3)}{\bvec{z}_{\cvec{G},T}}^2
    \lesssim \norm{L^2(\Omega)}{\Dh\uvec{w}_h}^2.
  \end{equation}
  Summing \eqref{eq:est.z.F} to \eqref{eq:est.z.GT} and recalling the definition \eqref{eq:tnorm.div.h} of $\tnorm{\DIV,h}{{\cdot}}$ along with \eqref{eq:z.cGT}  yields the inequality in \eqref{eq:Dh.inverse}.
\end{proof}


\section{Stability analysis of a DDR scheme for the magnetostatics problem} \label{sec:vectorLaplace}

We apply the Poincar\'e inequalities stated in Section \ref{sec:main.results} to the stability analysis of a DDR scheme for the magnetostatics problem which generalises the one presented in \cite{Di-Pietro.Droniou:21} to domains with non-trivial topology. 
We introduce the space of discrete harmonic forms 
\[
\HXdiv \coloneq \left\lbrace
\uvec{w}_h \in \Xdiv{h} \st
\text{%
  $\Dh \uvec{w}_h = 0$ and
  $(\uvec{w}_h, \uCh \uvec{v}_h)_{\DIV,h} = 0$ for all $\uvec{v}_h \in \Xcurl{h}$
}
\right\rbrace.
\]
For a given source term $\bvec{f} \in \bvec{H}^1(\Omega;\Real^3)$, we consider the following DDR approximation of the magnetostatics problem:
Find $(\uvec{\sigma}_h,\uvec{u}_h,\uvec{p}_h) \in \Xcurl{h} \times \Xdiv{h} \times \HXdiv$ such that, for all
$ (\uvec{\tau}_h,\uvec{v}_h,\uvec{q}_h) \in \Xcurl{h} \times \Xdiv{h} \times \HXdiv$, 
\[ 
A_h((\uvec{\sigma}_h,\uvec{u}_h,\uvec{p}_h), (\uvec{\tau}_h,\uvec{v}_h,\uvec{q}_h))
= (\uvec{I}^k_{\DIV,h} \bvec{f}, \uvec{v}_h)_{\DIV,h},
\] 
where the bilinear form $A_h:\big[\Xcurl{h} \times \Xdiv{h} \times \HXdiv\big]^2 \to \Real$ is given by:
\begin{equation} \label{eq:HL.defA}
  \begin{aligned}
    A_h((\uvec{\sigma}_h,\uvec{u}_h,\uvec{p}_h), (\uvec{\tau}_h,\uvec{v}_h,\uvec{q}_h))
    &\coloneq
    (\uvec{\sigma}_h, \uvec{\tau}_h)_{\CURL,h}
    - (\uvec{u}_h, \uCh \uvec{\tau}_h)_{\DIV,h}
    + (\uCh \uvec{\sigma}_h, \uvec{v}_h)_{\DIV,h}
    \\
    &\quad
    + (\Dh \uvec{u}_h, \Dh \uvec{v}_h)_{L^2(\Omega)}
    + (\uvec{p}_h,\uvec{v}_h)_{\DIV,h} + (\uvec{u}_h,\uvec{q}_h)_{\DIV,h},
  \end{aligned}
\end{equation}
with discrete $L^2$-like products $(\cdot,\cdot)_{\CURL,h}$ and $(\cdot,\cdot)_{\DIV,h}$ defined as in \cite[Section 4.4]{Di-Pietro.Droniou:23*1}.
We define the graph norm on this product space as
\begin{equation}\label{eq:norm.h}
  \norm{h}{(\uvec{\sigma}_h,\uvec{u}_h,\uvec{p}_h)}
  \coloneq \left(
  \norm{\CURL,h}{\uvec{\sigma}_h}^2
  + \norm{\DIV,h}{\uCh \uvec{\sigma}_h}^2
  + \norm{\DIV,h}{\uvec{u}_h}^2
  + \norm{L^2(\Omega)}{\Dh \uvec{u}_h}^2
  + \norm{\DIV,h}{\uvec{p}_h}^2
  \right)^{\frac12},
\end{equation}
with norms $\norm{\CURL,h}{{\cdot}} \coloneq (\cdot,\cdot)_{\CURL,h}^{\frac12}$ on $\Xcurl{h}$ and $\norm{\DIV,h}{{\cdot}} \coloneq (\cdot,\cdot)_{\DIV,h}^{\frac12}$ on $\Xdiv{h}$ induced by the corresponding discrete $L^2$-products.

\begin{theorem}[Stability of the discrete bilinear form]
  The discrete bilinear form \eqref{eq:HL.defA} is inf-sup stable for the graph norm $\norm{h}{{\cdot}}$, i.e., for all $(\uvec{\sigma}_h,\uvec{u}_h,\uvec{p}_h) \in \Xcurl{h} \times \Xdiv{h} \times \HXdiv$, it holds
  \begin{equation}\label{eq:inf-sup}
    \tnorm{h}{(\uvec{\sigma}_h,\uvec{u}_h,\uvec{p}_h)}
    \lesssim
    \sup_{(\uvec{\tau}_h,\uvec{v}_h,\uvec{q}_h) \in \Xcurl{h} \times \Xdiv{h} \times \HXdiv \setminus \{0\}}\frac{%
      A_h((\uvec{\sigma}_h,\uvec{u}_h,\uvec{p}_h), (\uvec{\tau}_h,\uvec{v}_h,\uvec{q}_h))
    }{%
      \tnorm{h}{(\uvec{\tau}_h,\uvec{v}_h,\uvec{q}_h)}
    }.
  \end{equation}
\end{theorem}

\begin{proof}
  Let $(\uvec{\sigma}_h,\uvec{u}_h,\uvec{p}_h)$ be given, 
  and let $C_{\rm P} > 0$ be the maximum of the hidden constants in the continuity estimates \eqref{eq:Dh.inverse} for the divergence and \eqref{eq:uCh.inverse} for the curl.
  
  Let $\uvec{z}_h \in \Xdiv{h}$ be given by \eqref{eq:Dh.inverse}, i.e., such that
  \begin{equation}\label{eq:zh}
    \text{%
      $\Dh \uvec{z}_h = \Dh \uvec{u}_h$ and $\norm{\DIV,h}{\uvec{z}_h} \leq C_{\rm P} \norm{L^2(\Omega)}{\Dh \uvec{u}_h} $.
    }
  \end{equation}
  Since $\Dh (\uvec{u}_h - \uvec{z}_h) = 0$, 
  there exists $(\uvec{\alpha}_h, \uvec{\beta}_h) \in \Xcurl{h} \times \HXdiv$ such that $\uvec{u}_h - \uvec{z}_h = \uCh \uvec{\alpha}_h + \uvec{\beta}_h$.
  Using \eqref{eq:uCh.inverse} applied to $\uvec{\alpha}_h$, we infer the existence of
  \begin{equation}\label{eq:alpha'h}
    \text{
      $\uvec{\alpha}_h'$ such that $\uCh \uvec{\alpha}_h' = \uCh \uvec{\alpha}_h$ and $\norm{\CURL,h}{\uvec{\alpha}_h'} \leq C_{\rm P} \norm{\DIV,h}{\uCh \uvec{\alpha}_h}$.
    }
  \end{equation}
  Noticing that %
    ${\norm{\DIV,h}{\uvec{u}_h - \uvec{z}_h}^2} = {\norm{\DIV,h}{\uCh \uvec{\alpha}_h}^2} + {\norm{\DIV,h}{\uvec{\beta}_h}^2}$ by orthogonality, using a triangle inequality we infer that
  \begin{equation} \label{eq:HL.bound.Cp}
    \norm{\DIV,h}{\uCh \uvec{\alpha}_h}^2
    + \norm{\DIV,h}{\uvec{\beta}_h}^2
    \lesssim \norm{\DIV,h}{\uvec{u}_h}^2
    + \norm{\DIV,h}{\uvec{z}_h}^2
    \overset{\eqref{eq:zh}}\le \norm{\DIV,h}{\uvec{u}_h}^2
    + C_{\rm P} \norm{L^2(\Omega)}{\Dh\uvec{u}_h}^2.
  \end{equation}

  We define
  \begin{equation} \label{eq:HL.def.TF}
    \uvec{\tau}_h \coloneq 2 C_{\rm P}^2 \uvec{\sigma}_h - \uvec{\alpha}_h', \quad
    \uvec{v}_h \coloneq \uCh \uvec{\sigma}_h + \uvec{p}_h + 2 C_{\rm P}^2 \uvec{u}_h , \quad 
    \uvec{q}_h \coloneq \uvec{\beta}_h - 2 C_{\rm P}^2 \uvec{p}_h.
  \end{equation}
  The following bound is readily inferred using triangle inequalities:
  \begin{equation} \label{eq:HL.bound.TF}
    \begin{aligned}
      &\norm{\CURL,h}{\uvec{\tau}_h}^2 
      + \norm{\DIV,h}{\uCh \uvec{\tau}_h}^2 
      + \norm{\DIV,h}{\uvec{v}_h}^2
      + \norm{L^2(\sigma)}{\Dh \uvec{v}_h}^2 
      + \norm{\DIV,h}{\uvec{q}_h}^2\\
      &\quad
      \begin{aligned}[t]
      &\lesssim 
      \norm{\CURL,h}{\uvec{\sigma}_h}^2 
      + \norm{\CURL,h}{\uvec{\alpha}_h'}^2 
      + \norm{\DIV,h}{\uCh \uvec{\sigma}_h}^2 
      + \norm{\DIV,h}{\uvec{p}_h}^2
      + \norm{\DIV,h}{\uvec{u}_h}^2
      + \norm{\DIV,h}{\uvec{\beta}_h}^2 \\
      \overset{\eqref{eq:alpha'h},\,\eqref{eq:HL.bound.Cp},\,\eqref{eq:norm.h}}&\lesssim       
      \tnorm{h}{(\uvec{\sigma}_h,\uvec{u}_h,\uvec{p}_h)}^2.
      \end{aligned}
    \end{aligned}
  \end{equation}
  Plugging the test functions \eqref{eq:HL.def.TF} into the expression \eqref{eq:HL.defA} of $A_h$ gives
  \begin{equation} \label{eq:HL.P.INFSUP.0}
    \begin{aligned}
      &A_h((\uvec{\sigma}_h,\uvec{u}_h,\uvec{p}_h), (\uvec{\tau}_h,\uvec{v}_h,\uvec{q}_h))
      \\
      &\quad=
      2 C_{\rm P}^2 \norm{\CURL,h}{\uvec{\sigma}_h}^2
      - (\uvec{\sigma}_h, \uvec{\alpha}_h')_{\CURL,h}
      \\
      &\qquad
      - \bcancel{2 C_{\rm P}^2 (\uvec{u}_h, \uCh \uvec{\sigma}_h)_{\DIV,h}}
       + (\uvec{u}_h, \uCh \uvec{\alpha}_h)_{\DIV,h} \\
      &\qquad
      + \norm{\DIV,h}{\uCh \uvec{\sigma}_h}^2 
      + \cancel{(\uCh \uvec{\sigma}_h, \uvec{p}_h)_{\DIV,h}} 
      + \bcancel{2 C_{\rm P}^2 (\uCh \uvec{\sigma}_h, \uvec{u}_h)_{\DIV,h}}
      \\
      &\qquad
      - (\Dh \uvec{u}_h, \cancel{\Dh \uCh \uvec{\sigma}_h})_{L^2(\Th)}
      - (\Dh \uvec{u}_h, \cancel{\Dh \uvec{p}_h})_{L^2(\Th)} 
      + 2 C_{\rm P}^2 \norm{L^2(\Omega)}{\Dh \uvec{u}_h}^2
      \\
      &\qquad
      + \cancel{(\uvec{p}_h,\uCh \uvec{\sigma}_h)_{\DIV,h} }
      + \norm{\DIV,h}{\uvec{p}_h}^2
      + \bcancel{2 C_{\rm P}^2 (\uvec{p}_h,\uvec{u}_h)_{\DIV,h}}
      \\
      &\qquad
      + (\uvec{u}_h,\uvec{\beta}_h)_{\DIV,h}
      - \bcancel{2 C_{\rm P}^2 (\uvec{u}_h,\uvec{p}_h)_{\DIV,h}} 
      \\
      &\quad=
      2 C_{\rm P}^2 \norm{\CURL,h}{\uvec{\sigma}_h}^2
      - (\uvec{\sigma}_h, \uvec{\alpha}_h')_{\CURL,h} 
      + \norm{\DIV,h}{\uCh \uvec{\sigma}_h}^2 + \norm{\DIV,h}{\uvec{p}_h}^2
      \\
      &\qquad
      + 2 C_{\rm P}^2 \norm{L^2(\Omega)}{\Dh \uvec{u}_h}^2 
      + (\uvec{u}_h, \uCh \uvec{\alpha}_h + \uvec{\beta}_h)_{\DIV,h}.
    \end{aligned}
  \end{equation}
  Using Cauchy--Schwarz and generalised Young inequalities, we have
  \begin{equation} \label{eq:HL.P.IN.1}
    (\uvec{\sigma}_h, \uvec{\alpha}_h')_{\CURL,h} 
    \overset{\eqref{eq:alpha'h}}\leq C_{\rm P} \norm{\CURL,h}{\uvec{\sigma}_h} \norm{\DIV,h}{\uCh\uvec{\alpha}_h} 
    \leq \frac{3 C_{\rm P}^2}{4} \norm{\CURL,h}{\uvec{\sigma}_h}^2 + \frac{1}{3} \norm{\DIV,h}{\uCh\uvec{\alpha}_h}^2.
  \end{equation}
  Moreover, the decomposition $\uvec{u}_h = \uvec{z}_h + \uCh \uvec{\alpha}_h + \uvec{\beta}_h$ gives 
  $(\uvec{u}_h, \uCh \uvec{\alpha}_h + \uvec{\beta}_h)_{\DIV,h} 
  = \norm{\DIV,h}{\uvec{u}_h}^2 - (\uvec{u}_h, \uvec{z}_h)_{\DIV,h}$ with
  \begin{equation} \label{eq:HL.P.IN.2}
    (\uvec{u}_h, \uvec{z}_h)_{\DIV,h}
    \leq  \norm{\DIV,h}{\uvec{u}_h} C_{\rm P}\norm{L^2(\Omega)}{\Dh\uvec{u}_h} 
    \leq \frac13 \norm{\DIV,h}{\uvec{u}_h}^2 + \frac{3C_{\rm P}^2}{4} \norm{L^2(\Omega)}{\Dh\uvec{u}_h}^2.
  \end{equation}
  Plugging \eqref{eq:HL.P.IN.1} and \eqref{eq:HL.P.IN.2} into \eqref{eq:HL.P.INFSUP.0}, we have
  \[
  \begin{aligned}
    &A_h((\uvec{\sigma}_h,\uvec{u}_h,\uvec{p}_h), (\uvec{\tau}_h,\uvec{v}_h,\uvec{q}_h))
    \\
    &\quad\geq 
    \frac{5 C_{\rm P}^2}{4} \norm{\CURL,h}{\uvec{\sigma}_h}^2 
    + \norm{\DIV,h}{\uCh \uvec{\sigma}_h}^2 + \norm{\DIV,h}{\uvec{p}_h}^2
    + \frac{5 C_{\rm P}^2}{4} \norm{L^2(\Omega)}{\Dh \uvec{u}_h}^2 
    + \frac23\norm{\DIV,h}{\uvec{u}_h}^2 - \frac13\norm{\DIV,h}{\uCh \uvec{\alpha}_h}^2
    \\
    &\quad\geq 
    \frac{5 C_{\rm P}^2}{4} \norm{\CURL,h}{\uvec{\sigma}_h}^2 
    + \norm{\DIV,h}{\uCh \uvec{\sigma}_h}^2 + \norm{\DIV,h}{\uvec{p}_h}^2
    + \frac{11 C_{\rm P}^2}{12} \norm{L^2(\Omega)}{\Dh \uvec{u}_h}^2 
    + \frac13\norm{\DIV,h}{\uvec{u}_h}^2
    \\
    &\quad\gtrsim
    \tnorm{h}{(\uvec{\sigma}_h,\uvec{u}_h,\uvec{p}_h)}^2.
  \end{aligned}
  \]
  Denoting by $\$$ the supremum in \eqref{eq:inf-sup}, we then use the previous bound to write
  \[
  \tnorm{h}{(\uvec{\sigma}_h,\uvec{u}_h,\uvec{p}_h)}^2
  \lesssim
  A_h((\uvec{\sigma}_h,\uvec{u}_h,\uvec{p}_h), (\uvec{\tau}_h,\uvec{v}_h,\uvec{q}_h))
  \le \$\, \tnorm{h}{(\uvec{\tau}_h,\uvec{v}_h,\uvec{q}_h)}
  \overset{\eqref{eq:HL.bound.TF}}\lesssim
  \$\, \tnorm{h}{(\uvec{\sigma}_h,\uvec{u}_h,\uvec{p}_h)}.
  \]
  Simplifying, the conclusion follows.
\end{proof}


\appendix

\section{Results on the trimmed finite element sequence on tetrahedral meshes}
\label{sec:simplicial.de-rham}

In this section we provide the explicit expression of the polynomial basis functions used in Section \ref{sec:proof.poincaré}.
These bases can be easily described on a reference element (see \cite{Arnold.Logg:14}).
However, in order the compute their norms, we need to know their expression on the physical element. 
The transformation from the reference to the physical element is given by the pullback of the mapping between the two.

We consider a simplex $S$ with 
  vertices $V_0$, $V_1$, $V_2$, and $V_3$ ordered so that, denoting by $\bvec{x}_{V_i}$ the coordinate vector of $V_i$,
$\left\{\bvec{x}_{V_1} - \bvec{x}_{V_0},\, \bvec{x}_{V_2} - \bvec{x}_{V_0},\, \bvec{x}_{V_3} - \bvec{x}_{V_0}\right\}$ forms a direct basis of $\Real^3$.
  The basis for the local affine Lagrange space $\Poly[1]{\trimmed}\Lambda^0(S) \cong \Poly{1}(S)$ spanned by ``hat'' functions is 
\begin{equation} \label{eq:S.DR.0}
  \begin{aligned}
    \phi_0(\bvec{x}) \coloneq 
    \frac{\big[(\bvec{x}_{V_2} - \bvec{x}_{V_1}) \times (\bvec{x}_{V_3} - \bvec{x}_{V_1}) \big] \cdot (\bvec{x}_{V_3} - \bvec{x})}
         {\det\big(\bvec{x}_{V_2} - \bvec{x}_{V_1}, \bvec{x}_{V_3} - \bvec{x}_{V_1}, \bvec{x}_{V_3}-\bvec{x}_{V_0}\big)}
    &,\;
    \phi_1(\bvec{x}) \coloneq 
    \frac{\big[(\bvec{x}_{V_2} - \bvec{x}_{V_0}) \times (\bvec{x}_{V_3} - \bvec{x}_{V_0}) \big] \cdot (\bvec{x} - \bvec{x}_{V_0})}
         {\det\big(\bvec{x}_{V_2} - \bvec{x}_{V_0}, \bvec{x}_{V_3} - \bvec{x}_{V_0}, \bvec{x}_{V_1}-\bvec{x}_{V_0}\big)}
    ,\\
    \phi_2(\bvec{x}) \coloneq 
    \frac{\big[(\bvec{x}_{V_1} - \bvec{x}_{V_0}) \times (\bvec{x}_{V_3} - \bvec{x}_{V_0}) \big] \cdot (\bvec{x} - \bvec{x}_{V_0})}
         {\det\big(\bvec{x}_{V_1} - \bvec{x}_{V_0}, \bvec{x}_{V_3} - \bvec{x}_{V_0}, \bvec{x}_{V_2}-\bvec{x}_{V_0}\big)}
    &,\;
    \phi_3(\bvec{x}) \coloneq 
    \frac{\big[(\bvec{x}_{V_1} - \bvec{x}_{V_0}) \times (\bvec{x}_{V_2} - \bvec{x}_{V_0}) \big] \cdot (\bvec{x} - \bvec{x}_{V_0})}
         {\det\big(\bvec{x}_{V_1} - \bvec{x}_{V_0}, \bvec{x}_{V_2} - \bvec{x}_{V_0}, \bvec{x}_{V_3}-\bvec{x}_{V_0}\big)}.
  \end{aligned}
\end{equation}
The basis of the lowest-order local face Raviart--Thomas--N\'ed\'elec space $\Poly[1]{\trimmed}\Lambda^1(S) \cong \RT{1}(S)$
is 
\begin{equation} \label{eq:S.DR.1}
  \begin{aligned}
    \bvec{\phi}_{23}(\bvec{x}) \coloneq 
    \frac{(\bvec{x}_{V_1} - \bvec{x}_{V_0}) \times (\bvec{x} - \bvec{x}_{V_0})}
         {\det\big(\bvec{x}_{V_1} - \bvec{x}_{V_0}, \bvec{x}_{V_2} - \bvec{x}_{V_0}, \bvec{x}_{V_3}-\bvec{x}_{V_2}\big)}
    &,\;
    \bvec{\phi}_{13}(\bvec{x}) \coloneq 
    \frac{(\bvec{x}_{V_2} - \bvec{x}_{V_0}) \times (\bvec{x} - \bvec{x}_{V_0})}
         {\det\big(\bvec{x}_{V_2} - \bvec{x}_{V_0}, \bvec{x}_{V_1} - \bvec{x}_{V_0}, \bvec{x}_{V_3}-\bvec{x}_{V_1}\big)}
    ,\\
    \bvec{\phi}_{12}(\bvec{x}) \coloneq 
    \frac{(\bvec{x}_{V_3} - \bvec{x}_{V_0}) \times (\bvec{x} - \bvec{x}_{V_0})}
         {\det\big(\bvec{x}_{V_3} - \bvec{x}_{V_0}, \bvec{x}_{V_1} - \bvec{x}_{V_0}, \bvec{x}_{V_2}-\bvec{x}_{V_1}\big)}
    &,\;
    \bvec{\phi}_{03}(\bvec{x}) \coloneq 
    \frac{(\bvec{x}_{V_2} - \bvec{x}_{V_1}) \times (\bvec{x} - \bvec{x}_{V_1})}
         {\det\big(\bvec{x}_{V_2} - \bvec{x}_{V_1}, \bvec{x}_{V_0} - \bvec{x}_{V_1}, \bvec{x}_{V_3}-\bvec{x}_{V_0}\big)}
    ,\\
    \bvec{\phi}_{02}(\bvec{x}) \coloneq 
    \frac{(\bvec{x}_{V_3} - \bvec{x}_{V_1}) \times (\bvec{x} - \bvec{x}_{V_1})}
         {\det\big(\bvec{x}_{V_3} - \bvec{x}_{V_1}, \bvec{x}_{V_0} - \bvec{x}_{V_1}, \bvec{x}_{V_2}-\bvec{x}_{V_0}\big)}
    &,\;
    \bvec{\phi}_{01}(\bvec{x}) \coloneq 
    \frac{(\bvec{x}_{V_3} - \bvec{x}_{V_2}) \times (\bvec{x} - \bvec{x}_{V_2})}
         {\det\big(\bvec{x}_{V_3} - \bvec{x}_{V_2}, \bvec{x}_{V_0} - \bvec{x}_{V_2}, \bvec{x}_{V_1}-\bvec{x}_{V_0}\big)},
  \end{aligned}
\end{equation}
  where the function $\bvec{\phi}_{ij}$ is associated to the edge $E_{ij}$ with vertices $V_i$ and $V_j$.
  The basis of the lowest-order local edge N\'ed\'elec space $\Poly[1]{\trimmed}\Lambda^2(S) \cong \NE{1}(S)$ is
\begin{equation} \label{eq:S.DR.2}
  \begin{aligned}
    \bvec{\phi}_{123}(\bvec{x}) \coloneq 
    \frac{2(\bvec{x} - \bvec{x}_{V_0})}
         {\det\big(\bvec{x}_{V_2} - \bvec{x}_{V_1}, \bvec{x}_{V_3} - \bvec{x}_{V_1}, \bvec{x}_{V_1}-\bvec{x}_{V_0}\big)}
    &,\;
    \bvec{\phi}_{023}(\bvec{x}) \coloneq 
    \frac{2(\bvec{x} - \bvec{x}_{V_1})}
         {\det\big(\bvec{x}_{V_2} - \bvec{x}_{V_0}, \bvec{x}_{V_3} - \bvec{x}_{V_0}, \bvec{x}_{V_1}-\bvec{x}_{V_0}\big)}
    ,\\
    \bvec{\phi}_{013}(\bvec{x}) \coloneq 
    \frac{2(\bvec{x} - \bvec{x}_{V_2})}
         {\det\big(\bvec{x}_{V_1} - \bvec{x}_{V_0}, \bvec{x}_{V_3} - \bvec{x}_{V_0}, \bvec{x}_{V_2}-\bvec{x}_{V_0}\big)}
    &,\;
    \bvec{\phi}_{012}(\bvec{x}) \coloneq 
    \frac{2(\bvec{x} - \bvec{x}_{V_3})}
         {\det\big(\bvec{x}_{V_1} - \bvec{x}_{V_0}, \bvec{x}_{V_2} - \bvec{x}_{V_0}, \bvec{x}_{V_3}-\bvec{x}_{V_0}\big)},
  \end{aligned}
\end{equation}
  with function $\bvec{\phi}_{ijk}$ associated to the simplicial face $F_{ijk}$ with vertices $V_i$, $V_j$, and $V_k$.
Finally, the basis of $\Poly[1]{\trimmed}\Lambda^3(S) \cong \Poly{0}(S)$ is 
\begin{equation} \label{eq:S.DR.3}
    \phi_{0123}(\bvec{x}) \coloneq 
    \frac{6}
         {\det\big(\bvec{x}_{V_1} - \bvec{x}_{V_0}, \bvec{x}_{V_2} - \bvec{x}_{V_0}, \bvec{x}_{V_3}-\bvec{x}_{V_0}\big)}
\end{equation}

\begin{lemma}[Dual basis]
  The following identities hold:
  \begin{subequations}
    \begin{alignat}{4}
      \phi_i(\bvec{x}_{i'}) &= \delta^i_{i'}
      &\qquad& \forall i, i' \in \lbrace 0, 1, 2, 3 \rbrace \label{eq:W.dual.0} \\
      \int_{E_{(ij)'}} \bvec{\phi}_{ij} \cdot \bvec{t}_E &= \delta^{ij}_{(ij)'} 
      &\qquad& \forall (ij), (ij)' \in \lbrace 23, 13, 12, 03, 02, 01 \rbrace \label{eq:W.dual.1}\\
      \int_{F_{(ijk)'}} \bvec{\phi}_{ijk} \cdot \bvec{n}_F &= \delta^{ijk}_{(ijk)'} 
      &\qquad&
      \forall (ijk),(ijk)' \in \lbrace 123, 023, 013, 012 \rbrace,\label{eq:W.dual.2}
    \end{alignat}
  \end{subequations}
  where $\delta_a^b = 1$ if $a = b$, $\delta_a^b = 0$ otherwise.
\end{lemma}

\begin{proof}
  The proof of \eqref{eq:W.dual.0} readily follows from the orthogonality of the cross product.

  Let us check \eqref{eq:W.dual.1} for $\bvec{\phi}_{23}$, the other being similar.
  For $(ij) \in \lbrace 23, 13, 12, 03, 02, 01 \rbrace$, we have
  \[
    \int_{E_{ij}} \bvec{\phi}_{23} \cdot \bvec{t}_E 
    = \int_{t=0}^1 \frac{
      \cancel{t [(\bvec{x}_{V_1} - \bvec{x}_{V_0}) \times (\bvec{x}_{V_j} - \bvec{x}_{V_i}) ] \cdot (\bvec{x}_{V_j} - \bvec{x}_{V_i})}
              + [(\bvec{x}_{V_1} - \bvec{x}_{V_0}) \times (\bvec{x}_{V_i} - \bvec{x}_{V_0})] \cdot (\bvec{x}_{V_j} - \bvec{x}_{V_i}) }
             {\det\big(\bvec{x}_{V_1} - \bvec{x}_{V_0}, \bvec{x}_{V_2} - \bvec{x}_{V_0}, \bvec{x}_{V_3}-\bvec{x}_{V_2}\big)}.
  \]
  If $i = 0$ or $i = 1$, then $(\bvec{x}_{V_1} - \bvec{x}_{V_0}) \times (\bvec{x}_{V_i} - \bvec{x}_{V_0}) = 0$.
  If $i = 2$ and $j = 3$, we can use the vector triple product to write
  $(\bvec{x}_{V_1} - \bvec{x}_{V_0}) \times (\bvec{x}_{V_2} - \bvec{x}_{V_0}) \cdot(\bvec{x}_{V_3} - \bvec{x}_{V_2}) = \det\big(\bvec{x}_{V_1} - \bvec{x}_{V_0}, \bvec{x}_{V_2} - \bvec{x}_{V_0}, \bvec{x}_{V_3}-\bvec{x}_{V_2}\big)$, 
  so that $\int_{E_{23}} \bvec{\phi}_{23} \cdot \bvec{t}_E = 1$.

  Finally, let us check \eqref{eq:W.dual.2} for $\bvec{\phi}_{123}$.
  Using the change of variable induced by $\bvec{\psi}:
  (\lambda_j,\lambda_k) \mapsto 
  \lambda_j (\bvec{x}_{V_j} - \bvec{x}_{V_i}) + 
  \lambda_k (\bvec{x}_{V_k} - \bvec{x}_{V_i}) + \bvec{x}_{V_i}$, %
    along with the fact that $\normal_F = \frac{(\bvec{x}_{V_j} - \bvec{x}_{V_i})\times (\bvec{x}_{V_k} - \bvec{x}_{V_i})}{2 |F_{ijk}|}$, we have
  \begin{equation*}
    \begin{aligned}
      &\int_{F_{ijk}} \bvec{\phi}_{123} \cdot \bvec{n}_F
      \\
      &\quad
      =
      \int_{\lambda_j = 0}^1 \int_{\lambda_k = 0}^{1 - \lambda_2} 
      2\frac{
        (\bvec{x}_{V_i} - \bvec{x}_{V_0}) + \lambda_2 (\bvec{x}_{V_j} - \bvec{x}_{V_i}) +  \lambda_3 (\bvec{x}_{V_3} - \bvec{x}_{V_i})
      }{
        \det\big(\bvec{x}_{V_2} - \bvec{x}_{V_1}, \bvec{x}_{V_3} - \bvec{x}_{V_1}, \bvec{x}_{V_1}-\bvec{x}_{V_0}\big)
      }
      \cdot \big[(\bvec{x}_{V_j} - \bvec{x}_{V_i}) \times (\bvec{x}_{V_k} - \bvec{x}_{V_i})\big]
      \\
      &\quad=
      2 \int_{\lambda_j = 0}^1 \int_{\lambda_k = 0}^{1 - \lambda_2} 
      \frac{
        (\bvec{x}_{V_i} - \bvec{x}_{V_0}) \cdot \big[(\bvec{x}_{V_j} - \bvec{x}_{V_i}) \times (\bvec{x}_{V_k} - \bvec{x}_{V_i})\big]
      }{
        \det\big(\bvec{x}_{V_2} - \bvec{x}_{V_1}, \bvec{x}_{V_3} - \bvec{x}_{V_1}, \bvec{x}_{V_1}-\bvec{x}_{V_0}\big)
      }
      \\
      &\quad=
      \frac{
        \det\big(\bvec{x}_{V_j} - \bvec{x}_{V_i}, \bvec{x}_{V_k} - \bvec{x}_{V_i},\bvec{x}_{V_i} - \bvec{x}_{V_0}\big)
      }{
        \det\big(\bvec{x}_{V_2} - \bvec{x}_{V_1}, \bvec{x}_{V_3} - \bvec{x}_{V_1}, \bvec{x}_{V_1}-\bvec{x}_{V_0}\big)
      }
      = \delta_{123}^{ijk}.\qedhere
    \end{aligned}
  \end{equation*}
\end{proof}

\begin{lemma}[Norm of the basis function]
  The functions given by \eqref{eq:S.DR.0}-\eqref{eq:S.DR.3} have the following $L^2$-norms:
  For all $i \in \lbrace 0, 1, 2, 3 \rbrace$,
  all $(ij) \in \lbrace 23, 13, 12, 03, 02, 01 \rbrace$,
  and all $(ijk) \in \lbrace 123, 023, 013, 012 \rbrace$,
  \begin{align} \label{eq:W.norm.0}
    \norm{L^2(S)}{\phi_i}^2
    &= \frac{\vert S \vert}{10} \simeq h_S^3,
    \\ \label{eq:W.norm.1}
    \norm{\bvec{L}^2(S;\Real^3)}{\bvec{\phi}_{ij}}^2
    &= \frac{
      \vert F_{012} \vert^2 + 
      \vert F_{013} \vert^2
      + c_{01}\vert F_{012} \vert\vert F_{013} \vert
    }{
      90 \vert S \vert
    }\simeq h_S,
    \\ \label{eq:W.norm.2}
    \norm{\bvec{L}^2(S;\Real^3)}{\bvec{\phi}_{ijk}}^2
    &=
    \frac{\vert E_{li} \vert^2 + \vert E_{lj} \vert^2 + \vert E_{lk} \vert^2 
      + 2 \vert F_{lij} \vert + 2 \vert F_{lik} \vert + 2 \vert F_{ljk}\vert}{180 \vert S \vert}\simeq h_S^{-1},
    \\ \label{eq:W.norm.3}
    \norm{L^2(S)}{\phi_{0123}}^2
    &= \frac{1}{\vert S \vert} \simeq h_S^{-3}, 
  \end{align}
  where, for $\lbrace i,j,k,l \rbrace = \lbrace 0, 1, 2, 3 \rbrace$, 
  $c_{kl} = \bvec{n}_{kli} \cdot \bvec{n}_{klj}$ is the dihedral angle associated to the edge $E_{kl}$.
\end{lemma}

\begin{proof}
  We will only show the computation for one function of each space, the others being similar.
  In order to integrate over the simplex S, we consider the change of variable 
  induced by $\bvec{\psi}: 
  (\lambda_1,\lambda_2,\lambda_3) \mapsto 
  \lambda_1 \bvec{x}_{V_1} + 
  \lambda_2 \bvec{x}_{V_2} + 
  \lambda_3 \bvec{x}_{V_3} + (1 - \lambda_1 - \lambda_2 - \lambda_3) \bvec{x}_{V_0}$.
  Notice that $\vert \det D \bvec{\psi} \vert = 6 \vert S \vert$.

  Let us first consider the family given by \eqref{eq:S.DR.0}.
  Using the orthogonality of the cross product, 
  and the identity $(\bvec{a} \times \bvec{b}) \cdot \bvec{c} = \det(\bvec{a},\bvec{b},\bvec{c})$,
  we notice that
  \[
  \phi_3 (\bvec{\psi}) = \lambda_3 \frac{\big[(\bvec{x}_{V_1} - \bvec{x}_{V_0}) \times (\bvec{x}_{V_2} - \bvec{x}_{V_0}) \big] \cdot (\bvec{x}_{V_3} - \bvec{x}_{V_0})}
      {\det\big(\bvec{x}_{V_1} - \bvec{x}_{V_0}, \bvec{x}_{V_2} - \bvec{x}_{V_0}, \bvec{x}_{V_3}-\bvec{x}_{V_0}\big)}
      = \lambda_3 .
      \]
      Hence, we have
      \[
      \int_S \phi_3^2
      = \int_{\lambda_1 = 0}^1 \int_{\lambda_2 = 0}^{1 - \lambda_1} \int_{\lambda_3 = 0}^{1 - \lambda_1 - \lambda_2} 
      (\lambda_3)^2 6 \vert S \vert
      = \frac{\vert S \vert}{10}.
      \]

      Then, we proceed with the family \eqref{eq:S.DR.1}.
      We have
      \begin{align*}
        \bvec{\phi}_{23}(\bvec{\psi}) =& 
        \frac{\lambda_2 (\bvec{x}_{V_1} - \bvec{x}_{V_0}) \times (\bvec{x}_{V_2} - \bvec{x}_{V_0}) 
          + \lambda_3 (\bvec{x}_{V_1} - \bvec{x}_{V_0}) \times (\bvec{x}_{V_3} - \bvec{x}_{V_0})}
             {\det\big(\bvec{x}_{V_1} - \bvec{x}_{V_0}, \bvec{x}_{V_2} - \bvec{x}_{V_0}, \bvec{x}_{V_3}-\bvec{x}_{V_2}\big)} \\
             =& \frac{1}{6 \vert S \vert} \left( \lambda_2 2 \vert F_{012} \vert \bvec{n}_{012} + 
             \lambda_3 2 \vert F_{013} \vert \bvec{n}_{013} \right).
      \end{align*}
      Expanding the product, we obtain
      \[
      \begin{aligned}
        \int_S \bvec{\phi}_{23} \cdot \bvec{\phi}_{23} ={}&
        \int_{\lambda_1 = 0}^1 \int_{\lambda_2 = 0}^{1 - \lambda_1} \int_{\lambda_3 = 0}^{1 - \lambda_1 - \lambda_2} 
        \left(\frac{1}{3 \vert S \vert}\right)^2 \left( \lambda_2^2 \vert F_{012} \vert^2 + 
        \lambda_3^2 \vert F_{013} \vert^2 + 2 \lambda_2 \lambda_3 c_{01}\vert F_{012} \vert\vert F_{013} \vert \right)
        6 \vert S \vert \\
        ={}& \frac{2}{3 \vert S \vert} \frac{\vert F_{012} \vert^2 + 
          \vert F_{013} \vert^2 + c_{01}\vert F_{012} \vert\vert F_{013} \vert}{60}
      \end{aligned}
      \]

      Finally, we prove that \eqref{eq:W.norm.2} holds for $\bvec{\phi}_{123}$ given by \eqref{eq:S.DR.2}.
      We have
      \[
      \bvec{\phi}_{123}(\bvec{\psi}) =
      2 \frac{\lambda_1 (\bvec{x}_{V_1} - \bvec{x}_{V_0})+ 
        \lambda_2 (\bvec{x}_{V_2} - \bvec{x}_{V_0}) + 
        \lambda_3 (\bvec{x}_{V_3} - \bvec{x}_{V_0})}{\det\big(\bvec{x}_{V_2} - \bvec{x}_{V_1}, \bvec{x}_{V_3} - \bvec{x}_{V_1}, \bvec{x}_{V_1}-\bvec{x}_{V_0}\big)} .
      \]
      Noticing that $(\bvec{x}_{V_i} - \bvec{x}_{V_0}) \cdot (\bvec{x}_{V_j} - \bvec{x}_{V_0}) = 2 \vert F_{0ij} \vert$,
      we have
      \[
      \begin{aligned}
        \int_S \bvec{\phi}_{123}\cdot\bvec{\phi}_{123} ={}&
        \int_{\lambda_1 = 0}^1 \int_{\lambda_2 = 0}^{1 - \lambda_1} \int_{\lambda_3 = 0}^{1 - \lambda_1 - \lambda_2} 
        \left(\frac{2}{6\vert S \vert}\right)^2
        \big(\lambda_1^2 \vert E_{01} \vert^2 + \lambda_2^2 \vert E_{02} \vert^2 + \lambda_3^2 \vert E_{03} \vert^2  \\
        &\qquad\qquad  + 4 \lambda_1\lambda_2 \vert F_{012} \vert + 4 \lambda_1\lambda_3 \vert F_{013} \vert 
        + 4 \lambda_2\lambda_3 \vert F_{023} \vert\big)
        6 \vert S \vert \\
        =& \frac{1}{3\vert S \vert} 
        \frac{\vert E_{01} \vert^2 + \vert E_{02} \vert^2 + \vert E_{03} \vert^2 
          + 2 \vert F_{012} \vert + 2 \vert F_{013} \vert + 2 \vert F_{023}\vert}{60}
      \end{aligned}
      \]
\end{proof}

\begin{lemma}[Link with the differential operators] \label{lem:W.diff}
  For all face $F$ of $S$, we define $\wSF \in \lbrace-1,1\rbrace$ such that $\wSF \nF$ is outward pointing.
  Then, the followings identities hold;
  \begin{alignat}{4} \label{eq:W.diff.0}
    \GRAD \phi_i &= \sum_{j < i} \bvec{\phi}_{ji} - \sum_{j > i} \bvec{\phi}_{ij}
    &\qquad& \forall i \in \lbrace 0,1,2,3 \rbrace,
    \\ \label{eq:W.diff.1}
    \CURL \bvec{\phi}_{ij} &= \omega_{SF_{\hat{i}}} \omega_{F_{\hat{i}}E_{ij}} \bvec{\phi}_{\hat{i}} + 
    \omega_{SF_{\hat{j}}} \omega_{F_{\hat{j}}E_{ij}} \bvec{\phi}_{\hat{j}}
    &\qquad& \forall (ij) \in \lbrace 23, 13, 12, 03, 02, 01 \rbrace,
    \\ \label{eq:W.diff.2} 
    \DIV \bvec{\phi}_{ijk} &= \omega_{SF_{ijk}} \phi_{0123},
    &\qquad& \forall (ijk) \in \lbrace 123, 023, 013, 012 \rbrace, 
  \end{alignat}
  where $\hat{i}$ denotes the complementary of $i$ in $\lbrace 0,1,2,3 \rbrace$.
\end{lemma}

\begin{proof}
  First, let us prove \eqref{eq:W.diff.2}.
  Noticing that $\DIV \bvec{x} = 3$, it only remain to check that 
  \begin{equation} \label{eq:P.W.diff.2}
    \sign \det\big(\bvec{x}_{V_j} - \bvec{x}_{V_i}, \bvec{x}_{V_k} - \bvec{x}_{V_i}, \bvec{x}_{V_i}-\bvec{x}_{V_l}\big) = \omega_{SF_{ijk}}.
  \end{equation}
  This holds, since $(\bvec{x}_{V_j} - \bvec{x}_{V_i}) \times (\bvec{x}_{V_k} - \bvec{x}_{V_i}) = 
  \Vert \bvec{x}_{V_j} - \bvec{x}_{V_i}\Vert \Vert \bvec{x}_{V_k} - \bvec{x}_{V_i} \Vert \nF$,
  and $\bvec{x}_{V_i}-\bvec{x}_{V_l}$ is always outward pointing.

  Then, to prove \eqref{eq:W.diff.1}, we use the identity $\CURL (\bvec{A}\times \bvec{B}) = \VDIV (\bvec{B} \otimes \bvec{A}^\top - \bvec{A} \otimes \bvec{B}^\top)$ in \eqref{eq:S.DR.1} to write 
  \begin{equation} \label{eq:P.W.diff.1.0}
    \CURL \bvec{\phi}_{ij} = \frac{3 (\bvec{x}_{V_l} - \bvec{x}_{V_k}) - (\bvec{x}_{V_l} - \bvec{x}_{V_k})}
          {\det \big(\bvec{x}_{V_l} - \bvec{x}_{V_k}, \bvec{x}_{V_i} - \bvec{x}_{V_k}, \bvec{x}_{V_j}-\bvec{x}_{V_i}\big)},
  \end{equation}
  where $(kl)$ is such that $\lbrace i,j,k,l \rbrace = \lbrace 0,1,2,3 \rbrace$ and $k < l$.
  Noticing that $\bvec{x}_{V_l} - \bvec{x}_{V_k}$ is an inward pointing to the face $\hat{k}$,
  and $\bvec{x}_{V_i} - \bvec{x}_{V_k}$ lies outward pointing in the plane of $F_{\hat{k}}$, 
  we have 
  \begin{equation} \label{eq:P.W.diff.1.1}
    \det \big(\bvec{x}_{V_l} - \bvec{x}_{V_k}, \bvec{x}_{V_i} - \bvec{x}_{V_k}, \bvec{x}_{V_j}-\bvec{x}_{V_i}\big)
    = \frac{1}{6 \vert S \vert} \omega_{F_{\hat{k}}E_{ij}} 
    = -\frac{1}{6 \vert S \vert} \omega_{F_{\hat{l}}E_{ij}}.
  \end{equation}
  where we inserted $\bvec{x}_{V_k} - \bvec{x}_{V_i}$ in the second argument of the determinant to get the second inequality.
  Inserting $\bvec{x} - \bvec{x}$ in \eqref{eq:P.W.diff.1.0} and replacing the denominator according to \eqref{eq:P.W.diff.1.1}, 
  we obtain
  $\CURL \bvec{\phi}_{ij} = 6 \vert S \vert \left( \omega_{F_{\hat{k}}E_{ij}} 2 (\bvec{x} - \bvec{x}_{V_k})
  + \omega_{F_{\hat{l}}E_{ij}} 2 (\bvec{x} - \bvec{x}_{V_l}) \right)
  $
  We infer \eqref{eq:W.diff.1} recalling \eqref{eq:P.W.diff.2}.

  Finally, let us prove \eqref{eq:W.diff.0}.
  We only prove the equality for $\phi_0$, the other three being similar.
  By the assumption on the basis, we have 
  $\det \big(\bvec{x}_{V_2} - \bvec{x}_{V_1}, \bvec{x}_{V_3} - \bvec{x}_{V_1}, \bvec{x}_{V_3}-\bvec{x}_{V_0}\big) = 6 \vert S \vert$.
  Then, a direct computation shows that
  \begin{align*}
    \GRAD \phi_0 
    =& -\frac{(\bvec{x}_{V_2} - \bvec{x}_{V_1})\times(\bvec{x}_{V_3} - \bvec{x}_{V_1})}{6 \vert S \vert}\\
    =& -\frac{(\bvec{x}_{V_2} - \bvec{x})\times(\bvec{x}_{V_3} - \bvec{x}_{V_2})
      + (\bvec{x}_{V_2} - \bvec{x})\times(\bvec{x}_{V_2} - \bvec{x}_{V_1})
      + (\bvec{x} - \bvec{x}_{V_1})\times(\bvec{x}_{V_3} - \bvec{x}_{V_1})
    }{6 \vert S \vert}.
  \end{align*}
  Noticing that
  \[
  \begin{gathered}
    \bvec{\phi}_{03}(\bvec{x}) = \frac{(\bvec{x}_{V_2} - \bvec{x}_{V_1}) \times (\bvec{x} - \bvec{x}_{V_1})}
         {6 \vert S \vert}
         ,\quad
         \bvec{\phi}_{02}(\bvec{x}) = -\frac{(\bvec{x}_{V_3} - \bvec{x}_{V_1}) \times (\bvec{x} - \bvec{x}_{V_1})}
              {6 \vert S \vert}
              ,
              \\
              \bvec{\phi}_{01}(\bvec{x}) = \frac{(\bvec{x}_{V_3} - \bvec{x}_{V_2}) \times (\bvec{x} - \bvec{x}_{V_2})}
                   {6 \vert S \vert},
  \end{gathered}
  \]
  we infer that $\GRAD \phi_0 = - \bvec{\phi}_{01} - \bvec{\phi}_{03} - \bvec{\phi}_{02}$.
\end{proof}


\section*{Acknowledgements}
Daniele Di Pietro acknowledges the partial support of \emph{Agence Nationale de la Recherche} through the grant ``HIPOTHEC''.
Both authors acknowledge the partial support of \emph{Agence Nationale de la Recherche} through the grant ANR-16-IDEX-0006 ``RHAMNUS''.


\printbibliography

@Article{         Amrouche.Ciarlet.ea:15,
  Author        = {Amrouche, C. and Ciarlet, P. G. and Mardare, C},
  Title         = {On a lemma of Jacques-Louis Lions and its relation to
                  other fundamental results},
  Journal       = {Journal de Mathématiques Pures et Appliquées},
  Volume        = {104},
  Year          = {2015},
  Pages         = {207--226},
  DOI           = {10.1016/j.matpur.2014.11.007}
}

@Misc{            Arnold.Logg:14,
  Author        = {D. N. Arnold and A. Logg},
  Title         = {Periodic table of the Finite Elements},
  Journal       = {SIAM News},
  Year          = {2014},
  Publisher     = {Society for Industrial {\&} Applied Mathematics ({SIAM})},
  Volume        = {47},
  Number        = {9},
  HowPublished  = {\url{http://www-users.math.umn.edu/~arnold/femtable/index.html}}
}

@Book{            Arnold:18,
  Author        = {Arnold, D.},
  Title         = {Finite Element Exterior Calculus},
  Publisher     = {SIAM},
  Year          = {2018},
  DOI           = {10.1137/1.9781611975543}
}

@Article{         Beirao-da-Veiga.Brezzi.ea:18,
  Author        = {Beir\~{a}o da Veiga, L. and Brezzi, F. and Dassi, F. and
                  Marini, L. D. and Russo, A.},
  Title         = {A family of three-dimensional virtual elements with
                  applications to magnetostatics},
  Journal       = {SIAM J. Numer. Anal.},
  Volume        = {56},
  Year          = {2018},
  Number        = {5},
  Pages         = {2940--2962},
  DOI           = {10.1137/18M1169886}
}

@Article{         Beirao-da-Veiga.Dassi.ea:22*1,
  Author        = {Beir\~{a}o da Veiga, L. and Dassi, F. and Manzini, G. and
                  Mascotto, L. D.},
  Title         = {Virtual elements for Maxwell's equations},
  Journal       = {Comput. Math. Appl.},
  Volume        = {116},
  Pages         = {82-99},
  Year          = {2022},
  DOI           = {10.1016/j.camwa.2021.08.019}
}

@Book{            Beirao-da-Veiga.Lipnikov.ea:14,
  Author        = {Beir\~ao da Veiga, L. and Lipnikov, K. and Manzini, G.},
  Title         = {The mimetic finite difference method for elliptic
                  problems},
  Series        = {MS\&A. Modeling, Simulation and Applications},
  Volume        = {11},
  Publisher     = {Springer, Cham},
  Year          = {2014},
  Pages         = {xvi+392},
  DOI           = {10.1007/978-3-319-02663-3}
}

@Misc{            Bonaldi.Di-Pietro.ea:23,
  Title         = {An exterior calculus framework for polytopal methods},
  Author        = {Bonaldi, F. and Di Pietro, D. A. and Droniou, J. and Hu,
                  K.},
  Month         = {3},
  Year          = {2023},
  EPrint        = {2303.11093},
  ArchivePrefix = {arXiv},
  PrimaryClass  = {math.NA},
  URL           = {http://arxiv.org/abs/2303.11093}
}

@Article{         Bonelle.Di-Pietro.ea:15,
  Author        = {Bonelle, J. and Di Pietro, D. A. and Ern, A.},
  Title         = {Low-order reconstruction operators on polyhedral meshes:
                  Application to {Compatible Discrete Operator} schemes},
  Journal       = {Computer Aided Geometric Design},
  Year          = {2015},
  Volume        = {35--36},
  Pages         = {27--41},
  DOI           = {10.1016/j.cagd.2015.03.015}
}

@Article{         Bonelle.Ern:14,
  Title         = {Analysis of compatible discrete operator schemes for
                  elliptic problems on polyhedral meshes},
  Author        = {Bonelle, J. and Ern, A.},
  Journal       = {ESAIM: Math. Model. Numer. Anal.},
  Volume        = {48},
  Pages         = {553--581},
  Year          = {2014},
  DOI           = {10.1051/m2an/2013104}
}

@Article{         Botti.Di-Pietro.ea:17,
  Author        = {Botti, M. and Di Pietro, D. A. and Sochala, P.},
  Title         = {A {Hybrid High-Order} method for nonlinear elasticity},
  Year          = {2017},
  Journal       = {SIAM J. Numer. Anal.},
  Volume        = {55},
  Number        = {6},
  Pages         = {2687--2717},
  DOI           = {10.1137/16M1105943}
}

@Article{         Brezzi.Buffa.ea:09,
  Author        = {Brezzi, F. and Buffa, A. and Lipnikov, K.},
  Title         = {Mimetic finite differences for elliptic problems},
  Journal       = {M2AN Math. Model. Numer. Anal.},
  Volume        = {43},
  Year          = {2009},
  Number        = {2},
  Pages         = {277--295},
  DOI           = {10.1051/m2an:2008046}
}

@Article{         Brezzi.Lipnikov.ea:05,
  Author        = {Brezzi, F. and Lipnikov, K. and Shashkov, M.},
  Title         = {Convergence of the mimetic finite difference method for
                  diffusion problems on polyhedral meshes},
  Journal       = {SIAM J. Numer. Anal.},
  Volume        = {43},
  Number        = {5},
  Year          = {2005},
  Pages         = {1872--1896},
  DOI           = {10.1137/040613950}
}

@Article{         Chave.Di-Pietro.ea:22,
  Author        = {Chave, F. and Di Pietro, D. A. and Lemaire, S.},
  Title         = {A discrete {Weber} inequality on three-dimensional hybrid
                  spaces with application to the {HHO} approximation of
                  magnetostatics},
  Journal       = {Math. Models Methods Appl. Sci.},
  Year          = {2022},
  Volume        = {32},
  Number        = {1},
  Pages         = {175--207},
  DOI           = {10.1142/S0218202522500051}
}

@Article{         Christiansen.Licht:20,
  Title         = {{Poincar{\'e}--Friedrichs inequalities of complexes of
                  discrete distributional differential forms}},
  Author        = {Christiansen, Snorre H and Licht, Martin W},
  Journal       = {BIT Numerical Mathematics},
  Volume        = {60},
  Number        = {2},
  Pages         = {345--371},
  Year          = {2020},
  DOI           = {10.1007/s10543-019-00784-1}
}

@Article{         Codecasa.Specogna.ea:10,
  Title         = {A new set of basis functions for the discrete geometric
                  approach},
  Author        = {Codecasa, L. and Specogna, R. and Trevisan, F.},
  Journal       = {J. Comput. Phys.},
  Number        = {299},
  Volume        = {19},
  Pages         = {7401--7410},
  Year          = {2010},
  DOI           = {10.1016/j.jcp.2010.06.023}
}

@Article{         Di-Pietro.Droniou.ea:23,
  title         = {Cohomology of the discrete {de Rham} complex on domains of
                  general topology},
  author        = {Di Pietro, D. A. and Droniou, J. and Pitassi, S.},
  year          = {2023},
  journal       = {Calcolo},
  doi           = {10.1007/s10092-023-00523-7},
  volume        = {60},
  number        = {32}
}

@Book{            Di-Pietro.Droniou:20,
  Author        = {Di Pietro, D. A. and Droniou, J.},
  Title         = {The {Hybrid High-Order} method for polytopal meshes},
  Subtitle      = {Design, analysis, and applications},
  Publisher     = {Springer International Publishing},
  Year          = {2020},
  Series        = {Modeling, Simulation and Application},
  Volume        = {19},
  DOI           = {10.1007/978-3-030-37203-3}
}

@Article{         Di-Pietro.Droniou:21,
  Author        = {Di Pietro, D. A. and Droniou, J.},
  Title         = {An arbitrary-order method for magnetostatics on polyhedral
                  meshes based on a discrete {de Rham} sequence},
  Year          = {2021},
  Journal       = {J. Comput. Phys.},
  Volume        = {429},
  Number        = {109991},
  DOI           = {10.1016/j.jcp.2020.109991}
}

@Article{         Di-Pietro.Droniou:23*1,
  Author        = {Di Pietro, D. A. and Droniou, J.},
  Title         = {An arbitrary-order discrete {de Rham} complex on
                  polyhedral meshes: Exactness, {Poincar\'e} inequalities,
                  and consistency},
  Journal       = {Found. Comput. Math.},
  Year          = {2023},
  Volume        = {23},
  Pages         = {85--164},
  DOI           = {10.1007/s10208-021-09542-8}
}

@Misc{            Di-Pietro.Hanot:23,
  Title         = {A discrete three-dimensional divdiv complex on polyhedral
                  meshes with application to a mixed formulation of the
                  biharmonic problem},
  Author        = {Di Pietro, D. A. and Hanot, M.-L.},
  Year          = {2023},
  Month         = {5},
  hal           = {hal-04093192},
  arxiv         = {2305.05729},
  EPrint        = {2305.05729},
  ArchivePrefix = {arXiv},
  PrimaryClass  = {math.NA}
}

@Misc{            Lemaire.Pitassi:23,
  Title         = {Discrete {Weber} inequalities and related {Maxwell}
                  compactness for hybrid spaces over polyhedral partitions of
                  domains with general topology},
  Author        = {Lemaire, S. and Pitassi, S.},
  Month         = {4},
  Year          = {2023},
  EPrint        = {2304.14041},
  ArchivePrefix = {arXiv},
  PrimaryClass  = {math.NA}
}

\end{document}